\documentclass[%
a4paper
]
{amsart}

\usepackage{amsmath,amssymb,amscd}
\usepackage{extarrows}

\usepackage{tikz-cd}

\numberwithin{equation}{section}
\newtheorem{theorem}[equation]{Theorem}
\newtheorem{claim}[equation]{Claim}
\newtheorem{corollary}[equation]{Corollary}
\newtheorem{lemma}[equation]{Lemma}

\newtheorem{proposition}[equation]{Proposition}

\theoremstyle{remark}
\newtheorem{remark}[equation]{Remark}
\theoremstyle{definition}
\newtheorem{definition}[equation]{Definition}

\newtheorem{example}[equation]{Example}

\renewcommand{\epsilon}{\varepsilon}
\renewcommand{\phi}{\varphi}

\newcommand{\Alg}{\ccat{Alg}}

\DeclareMathOperator{\Aut}{Aut}

\newcommand{\blank}{(\;)}
\newcommand{\boundary}{\partial}
\newcommand{\boundaryc}{\boundary\close}
\newcommand{\cardinality}{\mathrm{card}}
\newcommand{\cat}{\mathcal}
\newcommand{\Cat}{\ccat{Cat}}

\newcommand{\ccat}{\mathrm}
\newcommand{\Coalg}{\ccat{Coalg}}

\DeclareMathOperator*{\colim}{colim}
\DeclareMathOperator{\Conf}{Conf}

\newcommand{\widebar}{\overline}
\newcommand{\close}{\widebar}

\newcommand{\Com}{\ccat{Com}}
\newcommand{\complete}{\widehat}

\newcommand{\Cosh}{\ccat{Cosh}}
\newcommand{\cotensor}{\Box}

\newcommand{\diag}{\Delta}

\newcommand{\directsum}{\oplus}
\newcommand{\Dis}{\ccat{Disj}}

\newcommand{\Disk}{\ccat{Disk}}
\newcommand{\disj}{\sqcup}
\newcommand{\Disj}{\bigsqcup}
\renewcommand{\dot}{\bullet}

\newcommand{\Elu}{\mathbb{E}}

\DeclareMathOperator{\Emb}{Emb}

\renewcommand{\equiv}{\sim}

\newcommand{\equivwith}{\simeq}
\newcommand{\equivto}{\xrightarrow{\equiv}}

\DeclareMathOperator{\Fibre}{Fibre}
\newcommand{\field}{\mathbb}
\newcommand{\C}{\field{C}}
\newcommand{\R}{\field{R}}
\newcommand{\Z}{\field{Z}}
\newcommand{\Fin}{\ccat{Fin}}

\newcommand{\from}{\leftarrow}

\newcommand{\Ggamma}{\mathbf{\Gamma}}

\newcommand{\intersect}{\cap}

\newcommand{\interior}{\mathring}
\newcommand{\into}{\hookrightarrow}
\newcommand{\kara}{\varnothing}
\newcommand{\kore}{\textbf}

\DeclareMathOperator*{\laxcol}{laxcolim}

\newcommand{\Lie}{\ccat{Lie}}
\newcommand{\lie}{\mathfrak}
\newcommand{\lift}{\widetilde}
\newcommand{\loc}{\mathrm{loc}}
\newcommand{\longequivfrom}{\xlongleftarrow{\equiv}}

\newcommand{\longfrom}{\longleftarrow}
\newcommand{\longto}{\longrightarrow}

\newcommand{\Man}{\ccat{Man}}

\newcommand{\manc}{\ccat{Man}}
\newcommand{\Map}{\mathrm{Map}}
\newcommand{\Mod}{\ccat{Mod}}

\newcommand{\noudo}{\sharp}
\newcommand{\hitan}{\mathrm{nu}}

\newcommand{\onto}{\twoheadrightarrow}
\newcommand{\op}{\mathrm{op}}
\newcommand{\Open}{\ccat{Open}}

\newcommand{\positive}{+}
\newcommand{\pr}{\mathrm{pr}}
\newcommand{\Pre}{\ccat{Pre}}

\newcommand{\reduce}{\widetilde}

\newcommand*{\resto}[1]{|_{#1}}
\newcommand{\sphere}{\field{S}}

\newcommand{\Sh}{\ccat{Sh}}

\newcommand{\Simp}{\mathbf{\Delta}}
\newcommand{\simp}{\Delta}

\newcommand{\kocoprod}{\amalg}
\newcommand{\Space}{\ccat{Space}}

\newcommand{\sub}{\subset}

\newcommand{\susp}{\Sigma}
\DeclareMathOperator{\Sym}{Sym}

\newcommand{\tensor}{\otimes}
\newcommand{\Tensor}{\bigotimes}

\newcommand{\union}{\cup}
\newcommand{\Union}{\bigcup}

\newcommand{\unity}{\mathbf{1}}
\newcommand{\wreath}{\boldsymbol{\wr}}
\newcommand{\zero}{\mathbf{0}}

\title
{Koszul duality for locally constant factorisation algebras}
\author[Matsuoka Takuo]{Matsuoka, Takuo}
\email{motogeomtop@gmail.com}
\date{}


\begin{document}
\setcounter{section}{-1}
\setcounter{equation}{-1}
\maketitle

\begin{abstract}
Generalising Jacob Lurie's idea on the relation between the Verdier
duality and the iterated loop space theory, we study the Koszul
duality for locally constant factorisation algebras.
We formulate an analogue of Lurie's ``nonabelian Poincar\'e duality''
theorem (which is closely related to earlier results of Graeme Segal,
of Dusa McDuff, and of Paolo Salvatore) in a symmetric monoidal stable
infinity category carefully, using John Francis' notion of excision.
Its proof depends on our study of the Koszul duality for
$E_n$-algebras in \cite{local}.
As a consequence, we obtain a Verdier type equivalence for
factorisation algebras by a Koszul duality construction.
\end{abstract}

\tableofcontents

\section{Introduction}
\setcounter{subsection}{-1}
\setcounter{equation}{-1}

This paper is continuation of \cite{descent, local}, and together with
these papers, is based on the author's thesis \cite{thesis}.
Building on results from \cite{descent} on the foundations
for the theory of locally constant factorisation algebras, we study
here the Koszul duality for these algebras.
For this, we use the results of \cite{local} on the local case of the
Koszul duality for $E_n$-algebras.

\subsection{Koszul duality for factorisation algebras}
\label{sec:koszul-intro}
Lurie has discovered what he calls the ``nonabelian Poincar\'e
duality'' theorem \cite{higher-alg}.
(According to him, closely related results were earlier obtained by
Segal \cite{segal-config}, McDuff \cite{mcduff} and Salvatore
\cite{salvato}.)
Classically, the Poincar\'e duality theorem concerns a locally
constant sheaf of abelian groups, or more generally of
``\emph{stable}'' (in homotopical sense) objects such as chain
complexes or spectra, on a manifold.
Lurie's theorem states that a form of the theorem also holds with
\emph{unstable} coefficients, rather than stable or abelian
coefficients.
By ``unstable coefficients'', we mean the coefficients in a locally
constant sheaf of spaces.
One formulation of the classical Poincar\'e duality theorem is
that the compactly supported cohomology of the sheaf is a
\emph{homology} theory, namely, forms a cosheaf.
Lurie's discovery is that the suitable homology theory for unstable
coefficients is the topological chiral homology, which generalises the
cosheaf homology.
This homology theory determines a locally constant
factorisation algebra, rather than a cosheaf.
The following is a formulation of Lurie's theorem using the language
of factorisation algebras (and sheaves).

\begin{theorem}[Lurie]\label{thm:lurie-poincare}
Let $M$ be a manifold of dimension $n$.
Let $E$ be a locally constant sheaf of based spaces on $M$.
If every stalk of $E$ has connectivity at least $n$, then the
(locally constant) prealgebra $E^+$ of spaces on $M$ defined by
$E^+(U)=\Gamma_c(U,E)$, the compactly supported cohomology, is a
factorisation algebra.
\end{theorem}

He notes that the stalk of $E^+$ is the $n$-fold loop space of the
stalk of $E$, and the structure of a (pre)algebra of $E^+$ globalises
the $E_n$-algebra structure which characterises $n$-fold loop spaces
in the iterated loop space theory.
As a consequence, the theorem leads to a globalisation of the iterated
loop space theory in the form of an equivalence between suitable
infinity categories of sheaves and of factorisation algebras.
This is an unstable counterpart of the Verdier duality theorem
expressed as an equivalence between sheaves and cosheaves valued in a
stable infinity category.

\bigskip

Iterated loop space theory is an instance of the Koszul duality
\cite{gk} for $E_n$-algebras, and locally constant factorisation
algebras globalise $E_n$-algebras.
This motivates one to consider the Koszul duality for factorisation
algebras, and look for a generalisation of the Poincar\'e and the
Verdier theorems in this context.

For this purpose, we have defined \emph{compactly supported}
factorisation homology (Section
\ref{sec:compactly-supported-homology}).
Given a locally constant factorisation algebra $A$ on a manifold $M$,
we denote the compactly supported homology on an open submanifold $U$
by $\int^c_UA$.
The association $A^+\colon U\mapsto\int^c_UA$ is then a
pre\emph{co}algebra on $M$.
The question then is how close $A^+$ is to a factorisation coalgebra.

Unfortunately, there arises a problem very soon.
Namely, while topological chiral homology behaves very nicely in a
symmetric monoidal infinity category in which the monoidal
multiplication functor preserves sifted homotopy colimits
variablewise (see \cite{higher-alg,descent}), this assumption on
sifted colimits is essential.
Even though this condition is satisfied often in practice, if we would
like to also consider factorisation \emph{co}algebras, then we would
need the monoidal operations to also preserve sifted homotopy
\emph{limits}.
This is a very strong constraint, even though it is satisfied in
Lurie's context.

One of the principal aims of the present work is to remove this
constraint in some other contexts which arise in practice, at least
for the purpose of generalising Lurie's results.
In this direction, we have obtained quite satisfactory results by
restricting our attention to algebras which are \emph{complete} with
respect to a suitable filtration.
We have shown that many algebras which arise in practice are of this
kind (after some natural procedure of completion, if necessary).

Indeed, we have established in \cite{local} a very good theory of the
Koszul duality in the local case of $E_n$-algebras, in a
complete filtered context.
In this paper, we show that the desired global results
follow from this good local theory.
Let us overview the results of \cite{local} and then the main results
of the present work.

\subsection{Koszul duality for complete $E_n$-algebras}
Let $n$ be a finite non-negative integer.

Let us review the setting of \cite{local}.
Let $\cat{A}$ be a symmetric monoidal infinity category.
We assume that it has a filtration which is compatible with the
symmetric monoidal structure in a suitable way.
Primary examples are the category of filtered objects in a
reasonable symmetric monoidal infinity category and a symmetric
monoidal stable infinity category with a compatible t-structure
\cite{higher-alg} (satisfying a mild technical condition).
Another family of examples is given by functor categories
admitting the Goodwillie calculus \cite{goodwi}, where the filtration
is given by the degree of excisiveness.
Indeed, we also assume that $\cat{A}$ is stable in the sense stated
in Lurie's book \cite{higher-alg}.
(See To\"en--Vezzosi's \cite{toen-vezzosi} for the origin of the
notion.)
However, our monoidal structure is not the direct sum (which does not
have the kind of compatibility with the filtration we need), but is
like the (derived) tensor product of chain complexes, and the smash
product of spectra, so our context is ``nonabelian''.
We further assume that $\cat{A}$ is complete with respect to the
filtration in a suitable sense.
The mentioned examples admit completion, and in these examples,
the category $\cat{A}$ we indeed work in is the category of
complete objects in the mentioned category, equipped with completed
symmetric monoidal structure.
These categories satisfy a few further technical assumptions we need,
which we shall not state here.

In such a complete filtered infinity category $\cat{A}$, any algebra
comes with a natural filtration with respect to which it is complete.
In the mentioned examples, the towers associated to the filtration are
the canonical (or ``defining'') tower, the Postnikov tower, and the
Taylor tower, and the objects we deal with are the limits of the
towers.
We have established the Koszul duality for $E_n$-algebras in $\cat{A}$
which is \emph{positively} filtered.
The corresponding restriction on the filtration of coalgebras is given
by another condition which we call \kore{copositivity}.
The theorem is as follows.

\begin{theorem}[{\cite[Theorem 0.0]{local}}]
\label{thm:koszul-duality-e-n}
Let $\cat{A}$ be as above.
Then the constructions of Koszul duals give inverse equivalences
\[
\Alg_{E_n}(\cat{A})_\positive\xlongleftrightarrow{\equiv}\Coalg_{E_n}(\cat{A})_\positive
\]
between the infinity category of positive augmented $E_n$-algebras and
copositive augmented $E_n$-coalgebras in $\cat{A}$.
\end{theorem}

\subsection{The Poincar\'e and the Verdier theorems}
For obtaining global results from the local theory of \cite{local},
a breakthrough was the discovery by Francis of the
notion of \emph{excision} \cite{glanon,francis}.
Excision is concerned with what happens to the value of a prealgebra
when a manifold is glued as in the composition in a cobordism
category.
Namely, let $A$ be a prealgebra on a manifold $M$, and suppose
an open submanifold $U$ is cut into two pieces $V$ and $W$ along a
codimension $1$ submanifold $N$ whose normal bundle is trivialised.
Then one finds an $E_1$-algebra, which we shall denote by $A(N)$, by
restricting $A$ to a tubular neighbourhood of $N$ (diffeomorphic to
$N\times\R^1$ by the trivialisation) and then pushing it down to
$\R^1$.
One finds that $A(V)$ and $A(W)$ are right and left modules
respectively over $A(N)$, and the \kore{excision} property requires
that the canonical map
\[
A(V)\tensor_{A(N)}A(W)\longto A(U)
\]
be an equivalence in every such situation (where the tensor product
should be understood as ``derived'', if there is also an underived,
i.e., not homotopy invariant, tensor product).

Francis proved that excision property characterises topological chiral
homology, and have applied this theorem to give a simple proof of
Lurie's nonabelian Poincar\'e duality theorem \cite{glanon,francis}.
The excision property gives a convenient way to compute topological
chiral homology.

Influenced by this work, we \emph{formulate} the Poincar\'e
duality theorem in our context using excision, and in this form, the
theorem holds if the local theory is good enough.
We say that an augmented locally constant factorisation algebra on an
$n$-dimensional manifold is \kore{positive} if it is locally so as an
$E_n$-algebra.

\begin{theorem}[Theorem \ref{thm:poincare-duality-complete}]
\label{thm:poincare-intro}
Let $A$ be a positive augmented locally constant factorisation
algebra on $M$, valued in $\cat{A}$ as in the previous section.
Then the precoalgebra $A^+$ defined by $A^+(U)=\int^c_UA$, satisfies
excision.
\end{theorem}

In order to show the ubiquity of algebras to which this theorem
applies, we have shown that \emph{any} augmented factorisation algebra
(taking values in a reasonable symmetric monoidal stable infinity
category) comes with a canonical positive filtration
(\textbf{Proposition \ref{prop:filtered-factorisation-algebra}}).
It follows that the Poincar\'e theorem holds for its completion
(Corollary \ref{cor:poincare-for-completed-algebra}).

See Section \ref{sec:variant} for another example.

\begin{remark}
Related results can be found in the work of
Francis \cite{francis}, and Ayala and Francis \cite{ayala-fran}.
\end{remark}

\smallskip

Theorem \ref{thm:poincare-intro} also leads to the Koszul duality for
factorisation algebras, as a Verdier type equivalence of categories.
Let $\Alg_M(\cat{A})_\positive$ (resp.~$\Coalg_M(\cat{A})_\positive$)
denote the infinity category of positive (resp.~copositive) augmented
prealgebras (resp.~precoalgebras) on $M$ which is locally constant in
a suitable sense, and satisfies excision.

\begin{theorem}[Theorem \ref{thm:verdier}]
\label{thm:verdier-intro}
Let $\cat{A}$ be as above.
Then the functor
\[
\blank^+\colon\Alg_M(\cat{A})_\positive\longto\Coalg_M(\cat{A})_\positive
\]
is an equivalence.
\end{theorem}

\subsection{Outline}
\textbf{Section \ref{sec:terminology-notation}} is for introducing
conventions which are used throughout the main body.

In \textbf{Section \ref{sec:generalise-apply}}, we discuss excision.

In \textbf{Section \ref{sec:compactly-supported-homology}}, we
introduce compactly supported factorisation homology, and investigate
its symmetric monoidal functoriality in manifolds.

In \textbf{Section \ref{sec:duality}}, we formulate the Poincar\'e
duality for a factorisation algebra as a relation between the
compactly supported factorisation homology and the Koszul dual of the
factorisation algebra.
We also investigate this for factorisation algebras in the situation
opposite to Lurie's.
This is another situation where problem about sifted limits does not
arise.

We then prove the
Poincar\'e duality theorem for complete factorisation algebras.
We compare the cases of (globally) constant algebras of this theorem,
with an implication of the Morita theoretic functoriality of the
Koszul duality \cite[Theorem 4.22]{local} on topological field
theories.
We also discuss particular algebras to which the theorem applies,
including one which is of interest from quantum field theory in
Costello--Gwilliam's framework \cite{cg} (see
\textbf{Theorem \ref{thm:sheaf-of-lie}} for the result).

\subsection*{Acknowledgment}
This paper is based on part of the author's thesis \cite{thesis}.
I am particularly grateful to my advisor Kevin Costello for his
extremely patient guidance and continuous encouragement and support.
My contribution through this work to the subject of factorization
algebra can be understood as technical work of combining the ideas and
work of the pioneers such as Jacob Lurie, John Francis, and Kevin.
I am grateful to those people for their work, and for making
their ideas accessible.
Special thanks are due to John for detailed comments and
suggestions on the drafts of my thesis, which were essential for
many improvements of both the contents and exposition.
Many of those improvements were inherited by this paper.
I am grateful to him also for
explaining to me the relation between the Poincar\'e duality theorems
in Lurie's context and in its opposite context.
I am grateful to Owen Gwilliam, Josh Shadlen, and Yuan Shen for their
continuous encouragement.

\section{Terminology, notations and conventions}
\label{sec:terminology-notation}
\setcounter{subsection}{-1}
\setcounter{equation}{-1}
\subsection{Category theory}
\label{sec:category-theory}

As in \cite{descent,local}, we adopt the
following convention for the terminology.

By a \kore{$1$-category}, we always mean an \emph{infinity}
$1$-category.
We often call a $1$-category (namely an infinity $1$-category) simply
a \kore{category}.
A category with discrete sets of morphisms (namely, a ``category''
in the more traditional sense) will be called $(1,1)$-category, or a
\emph{discrete} category.

In fact, all categorical and algebraic terms will be used in
\emph{infinity} ($1$-) categorical sense without further notice.
Namely, categorical terms are used in the sense enriched in the
\emph{infinity} $1$-category of spaces, or equivalently, of infinity
groupoids, and algebraic terms are used freely in the sense
generalised in accordance with the enriched categorical structures.

For example, for an integer $n$, by an \emph{$n$-category}
(resp.~\emph{infinity} category), we mean an \emph{infinity}
$n$-category (resp.~infinity infinity category).
We also consider multicategories.
By default, multimaps in our multicategories will form
a \emph{space} with all higher homotopies allowed.
Namely, our ``\emph{multicategories}'' are ``infinity operads'' in the
terminology of Lurie's book \cite{higher-alg}.

\begin{remark}
We usually treat a space relatively to the structure of the standard
(infinity) $1$-category of spaces.
Namely, a ``\emph{space}'' for us is usually no more than an object of
this category.
Without loss of information, we shall freely identify a space in this
sense with its fundamental infinity groupoid, and call it also a
``\emph{groupoid}''.
Exceptions in which the term ``space'' means not necessarily
this, include a ``Euclidean space'', the ``total space'' of a fibre
bundle, etc., in accordance with the common customs.
\end{remark}

\smallskip

The following notations and terminology will be used as in
\cite{descent,local}.

We use the following notations for over and under categories.
Namely, if $\cat{C}$ is a category and $x$ is an object of $\cat{C}$,
then we denote the category of objects $\cat{C}$ lying over $x$, i.e.,
equipped with a map to $x$, by $\cat{C}_{/x}$.
We denote the under category for $x$, in other words,
$((\cat{C}^\op)_{/x})^\op$, by $\cat{C}_{x/}$.

More generally, if a category $\cat{D}$ is equipped with a functor to
$\cat{C}$, then we define
$\cat{D}_{/x}:=\cat{D}\times_{\cat{C}}\cat{C}_{/x}$, and similarly for
$\cat{D}_{x/}$.
Note here that $\cat{C}_{/x}$ is mapping to $\cat{C}$ by the functor
which forgets the structure map to $x$.
Note that the notation is abusive in that the name of the functor
$\cat{D}\to\cat{C}$ is dropped from it.
In order to avoid this abuse from causing any confusion, we shall use
this notation only when the functor $\cat{D}\to\cat{C}$ that we are
considering is clear from the context.

\bigskip

By the \kore{lax colimit} of a diagram of categories indexed by a
category $\cat{C}$, we mean the Grothendieck construction.
We choose the variance of the laxness so the lax colimit projects to
$\cat{C}$, to make it an op-fibration over $\cat{C}$, rather
than a fibration over $\cat{C}^\op$.
(In particular, if $\cat{C}=\cat{D}^\op$, so the functor is
contravariant on $\cat{D}$, then the familiar fibred category over
$\cat{D}$ is the \emph{op}-lax colimit over $\cat{C}$ for us.)
Of course, we can choose the variance for lax \emph{limits} compatibly
with this, so our lax colimit generalises to that in any $2$-category.

\subsection{Symmetric monoidal structure}
\label{sec:symmetric-monoidal}

The following explicit definition of a symmetric
monoidal category will be used.
Namely, we follow To\"en \cite{toen-tannaka} to define a symmetric
monoidal (infinity $1$-) category as an infinity $1$-categorical
generalisation of Segal's \emph{$\Gamma$-category}
\cite{segal-cohomology}, in accordance with Lurie's book
\cite[Definitions 2.0.0.7, 2.1.3.7]{higher-alg}.
Namely, let $\Fin_*$ denote the category of pointed finite sets
(equivalently, the opposite of Segal's category $\Ggamma$).
For a finite set $S$, denote by $S_+=S\kocoprod\{*\}$ the object of
$\Fin_*$ obtained by externally adding a base point ``$*$'' to $S$.
\begin{definition}
Let $\Cat$ denote the ($2$-)category of categories (with some fixed
limit for their sizes).
Let $\cat{C}$ be a category, i.e., an object of $\Cat$.
Then a \kore{pre-$\Gamma$-structure} on $\cat{C}$ consists of a
functor $\cat{A}\colon\Fin_*\to\Cat$ together with an equivalence
$\cat{C}\equivwith\cat{A}(\sphere^0)$, where $\sphere^0$ denotes the two pointed
set with one base point.

A pre-$\Gamma$-structure as above is a \kore{symmetric monoidal}
structure if for every finite set $S$ (including the case $S=\kara$),
Segal's map
\begin{equation}\label{eq:segal-map}
\cat{A}(S_+)\longto\cat{A}(\sphere^0)^S
\end{equation}
is an equivalence \cite[Definition 2.1]{segal-cohomology}.

A \kore{pre-$\Gamma$-category} is a category equipped with a
pre-$\Gamma$-structure, or equivalently, just a functor
$\Fin_*\to\Cat$.
It is a \kore{symmetric monoidal} category if the
pre-$\Gamma$-structure is in fact a symmetric monoidal structure.

The category of \emph{maps} (``\kore{symmetric monoidal functors}'')
between symmetric monoidal categories is by definition, the category
of maps of the functors on $\Fin_*$.
\end{definition}

Let $\cat{A}$ be a symmetric monoidal category (i.e., a
pre-$\Gamma$-category $\Fin_*\to\Cat$ satisfying the required
condition).
Then through the equivalence \eqref{eq:segal-map}, the map
\begin{equation}\label{eq:partial-operation}
\cat{A}(S_+)\longto\cat{A}(\sphere^0)
\end{equation}
induced from the map which collapses
$S$ to the (non-base) point can be considered as a functor
\begin{equation}\label{eq:monoidal-multiplication}
\cat{C}^S\longto\cat{C},
\end{equation}
where $\cat{C}$ is the underlying category $\cat{A}(\sphere^0)$.
These can be considered as ``multiplication'' operations on $\cat{C}$
which results from the symmetric monoidal structure.
In fact, since $\cat{A}(S_+)$ can be replaced by $\cat{C}^S$, so
Segal's maps will be the identities, a symmetric monoidal structure on
$\cat{C}$ amounts to the operations \eqref{eq:monoidal-multiplication}
together with suitable compatibility data among them.

In our notation, we often use the same symbol for a symmetric monoidal
category and its underlying category, when this seems to cause no
confusion.
On the other hand, the name for a symmetric monoidal structure will
often be something like ``$\tensor$'', in which case the name of the
multiplication operation \eqref{eq:monoidal-multiplication} will be
$\Tensor_S$.
(If the operations already have names such as $\Tensor_S$, we will
name the symmetric monoidal structure after them, so the stated
rule will apply.)

\bigskip

We shall also need to consider \emph{partially defined} monoidal
structures.
The following simple definition will suffice for our purposes.
\begin{definition}
A \kore{partial} symmetric monoidal category is a
pre-$\Gamma$-category $\cat{A}$ for which Segal's maps
\eqref{eq:segal-map} are fully faithful functors.

The category of \emph{maps} (``\kore{symmetric monoidal functors}'')
of partial symmetric monoidal categories is by definition, the
category of maps of the functors on $\Fin_*$.
\end{definition} 

In a partial symmetric monoidal category $\cat{A}$,
\eqref{eq:partial-operation} is a multiplication operation
defined only on the full subcategory $\cat{A}(S_+)$ of $\cat{A}^S$.
We shall often denote $\cat{A}(S_+)$ by $\cat{A}^{(S)}$, so Segal's
map will be $\cat{A}^{(S)}\into\cat{A}^S$, while the multiplication
will be $\cat{A}^{(S)}\to\cat{A}$.

\subsection{Manifolds and factorisation algebras}
\label{sec:manifold-algebra}

In this paper, every \kore{manifold without boundary} is assumed to
be the interior of a \emph{specified} smooth compact manifold with
(possibly empty) boundary.
Namely, such a manifold $U$ comes \emph{equipped} with a smooth
compact closure which will be usually denote by $\close{U}$.
By an \kore{open embedding} $U\into V$ of such manifolds, where $U$
and $V$ are specified as the interior of compact $\close{U}$ and
$\close{V}$ respectively, we mean an open embedding $U\into V$ in the
usual sense which extends to a smooth immersion
$\close{U}\to\close{V}$.
By definition an \kore{open submanifold} is a manifold embedded in
this sense.

There will be a switch in notations from \cite{descent} accordingly.
Firstly, for a manifold $M$ without boundary, $\Open(M)$ introduced
in \cite[Section2.0]{descent}
will now denote the open submanifolds in the above sense.
Note that by \cite[Corollary 2.39, Example 2.40]{descent},
this class of manifolds are sufficient for understanding locally
constant factorisation algebras.
From the standpoint of the original conventions, this means that we
work only with prealgebras which are left Kan extensions of
their restriction to this class of manifolds (but the category of
locally constant factorisation algebras remains unchanged).

Similarly, in this paper, $\Disk(M)$ and $\Dis(M)$ used in
\cite[Section 2.0]{descent}
following Lurie \cite{higher-alg}, contain as objects (disjoint unions
of) disks which are diffeomorphic to the interior of the standard
closed disk, and is embedded in $M$ in the above sense.
All results of \cite{descent} are valid under these switched
notations.

\bigskip

In this paper, we assume as in \cite{descent} that the target category
$\cat{A}$ of prealgebras has sifted colimits, and the monoidal
multiplication functor on $\cat{A}$ preserves sifted colimits
variable-wise.
Equivalently, the monoidal multiplication should preserve sifted
colimits for all the variables at the same time.

By a (symmetric) \kore{monoidal} structure on a \emph{stable}
category, we mean a (symmetric) monoidal structure on the underlying
category for which the monoidal multiplication functors are
\emph{exact} in each variable.
Note that this and the above implies that when we consider a symmetric
monoidal stable category $\cat{A}$ for the target of prealgebras,
$\cat{A}$ is closed under all colimits, and the monoidal
multiplication functors preserve all colimits variable-wise.

\bigskip

For all other notations and terminology about factorisation algebras,
we follow \cite{descent}.

\section{Excision property of a factorisation algebra}
\label{sec:generalise-apply}
\setcounter{subsection}{-1}
\setcounter{equation}{-1}

\subsection{Constructible algebra on a closed interval}

Let $I$ be a closed interval.
Let $\Open(I)$ be the poset of open subsets of $I$.
This has a partially defined symmetric monoidal structure
given by taking disjoint union.
A \kore{prealgebra} on $I$ is defined to be a symmetric monoidal
functor on $\Open(I)$.

As a manifold with boundary, $I$ has a natural stratification given by
$\boundary I\sub I$.
We shall define the class of prealgebras on $I$ which we shall call
\emph{constructible} factorisation algebras, where the
constructibility is with respect to the mentioned stratification of $I$.

Let $\Disk(I)$ be the full subcategory of the poset
$\Open(I)$ consisting of objects of $\Disk(I-\boundary I)$ and collars
of either point of $\boundary I$.
This is a symmetric multicategory by inclusion of disjoint unions.

Let us say that a functor defined on the underlying poset (of
``colours'') of $\Disk(I)$ is \kore{constructible} if it inverts
morphisms from $\Disk(I-\boundary I)$ and morphisms between collars of
points of $\boundary I$.

\begin{definition}
Let $\cat{A}$ be a symmetric monoidal category.

Then a prealgebra on $I$ in $\cat{A}$ is said to be
\kore{constructible} (with respect to the stratification of $I$ as a
manifold with boundary) if its restriction to $\Disk(I)$ is
constructible.
Let us denote by $\Pre\Alg_I(\cat{A})$ the category of
\emph{constructible} prealgebras on $I$.

A \kore{constructible factorisation algebra} (or just
``\kore{constructible algebra}'') on $I$ is an algebra on $\Disk(I)$
whose underlying functor on colours is constructible.
The category of constructible algebras on $I$ in $\cat{A}$ will be
denoted by $\Alg_I(\cat{A})$.
\end{definition}

In fact, we can identify the category of constructible algebra on
$I$ with a right localisation of the category of constructible
prealgebras consisting of those prealgebras whose underlying functor
is a left Kan extension from a certain full subcategory, denoted by
$\Dis(I)$, of $\Open(I)$.
Namely, this full subcategory $\Dis(I)$ is the smallest full
subcategory of $\Open(I)$ containing $\Disk(I)$, and is closed under
the partial monoidal structure of $\Open(I)$, the disjoint union
operation.

In fact, the only interesting open submanifold of $I$ is $I$ itself,
and we have the following.

\begin{lemma}
Let $A$ be a constructible prealgebra on $I$.
Then $A$ is a (constructible) factorisation algebra if and only if the
map $\colim_{\Dis(I)}A\to A(I)$ is an equivalence.
\end{lemma}

We will next see that this colimit can be calculated as a tensor
product in the following way.
Given a constructible (pre-)algebra on $I$, note that the values of
$A$ on any object of $\Disk(I-\boundary I)$ are canonically
equivalent to each other (since they are all canonically equivalent to
$A(I-\boundary I)$).
Similarly, the values of $A$ on any collar of left end point of $I$ is
canonically equivalent to each other (since these collars are
totally ordered), and similarly around the right end point.
Denote these objects by $B$, $K$, $L$ respectively.
Then we want to see in particular, that there functorially exists a
structure of associative algebra on $B$, a structure of its right
module on $K$, a structure of its left module on $L$ for which there
is a natural map $K\tensor_BL\to A(I)$.

\begin{remark}
Moreover, there will be a natural right $B$-module map $B\to K$, and a
left $B$-module map $B\to L$.
Naturally, this can be understood as that $K$ (resp.~$L$) is an
$E_0$-algebra in the category of right (resp.~left) $B$-modules.
(This is a factorisation algebra on a point which associates the
object $B$ to the empty set.)

There is an obvious way to modify the definition of a constructible
prealgebra so that these extra structures will not come with the
structure of $A$.
\end{remark}

In order to do this, we extend isotopy invariance result from
\cite[Section 2.3]{descent}
to the present (actually very simple) context.
Let $\Elu_I$ denote the multicategory which has the same objects as
$\Disk(I)$, but the space of multimaps $\{U_i\}_i\to V$ in $\Elu_I$ is
the space formed by pairs consisting of an embedding
$f\colon\coprod_iU_i\into V$ and an isotopy of each $U_i$ in $I$ from
the defining inclusion $U_i\into I$ to $f\resto{U_i}$.

For distinction between a multicategory and its underlying category
(of ``colours''),
let us denote by $\Elu_{1,I}$ the underlying category of $\Elu_I$.
Then $\Elu_{1,I}$ is equivalent to the poset of subsets of $\boundary I$
consisting of \emph{at most} one element.

There is a morphism $\Disk(I)\to \Elu_I$ of multicategories, and
clearly, the underlying functor $\Disk_1(I)\to \Elu_{1,I}$, where we
have put the subscript $1$ for distinction, is a localisation
inverting inclusions in $\Disk_1(I-\boundary I)$ and inclusions of
collars of a point of $\boundary I$ (namely, those morphisms which are
required to be inverted by a constructible prealgebra).
Note that this is how we found the objects $B$, $K$, $L$ above.

In particular, if $A$ is a prealgebra on $I$, then its
underlying functor extends to a functor on $\Elu_{1,I}$ if and only if
$A$ is constructible.
Moreover, the extension is unique.

$\Elu_I$ is slightly more involved than $\Elu_{1,I}$, but it is still
homotopically discrete, and here is a complete description of it:
Given a functor on $\Elu_{I,1}$, a structure on it of an algebra on
$\Elu_I$, is exactly a structure of associative algebra on $B$,
a structure of right $B$-module on $K$, a structure of left $B$-module
map on the map $B\to K$, a structure of a left $B$-module on $L$, a
structure of a left $B$-module map on the map $B\to L$.

With this description available, a direct inspection shows that the
restriction through the morphism $\Disk(I)\to\Elu_I$ induces an
equivalence between the category of constructible algebra on $I$, and
the category of algebras over $\Elu_I$.
(For example, we could consider as an intermediate step, a symmetric
multicategory defined similarly to $\Elu_I$, but using only isotopies
(and their isotopies etc.)~which are piecewise linear.)

Let us now introduce a similarly `localised' version of $\Dis(I)$.
Namely we define a suitable variant of Lurie's $\ccat{D}(M)$
\cite{higher-alg}.

Let $\Man^1$ be the following category.
Namely, its object is a $1$-dimensional manifold with boundary which
is a finite disjoint union (coproduct) of open or half-open intervals.
The space of morphisms is the space of embeddings each of which sends
any boundary point to a boundary point.

Define $\ccat{D}(I):=\Man^1_{/I}$.
Its objects are open submanifolds of $I$ which are homeomorphic to a
finite disjoint union of disks.
The space of maps $U\to V$ is the space formed by embeddings $f\colon
U\into V$ together with an isotopy from the defining
inclusion $U\into I$ to $f\colon U\into I$.

\begin{lemma}\label{lem:disj-cofinal-d}
The functor $\Dis(I)\to\ccat{D}(I)$ is cofinal.
\end{lemma}
\begin{proof}
Proof is similar to the proof of \cite[Proposition 5.3.2.13
(1)]{higher-alg}.

We apply Joyal's generalisation of Quillen's theorem A \cite{topos}.

Let $D\in\Man^1$ with an embedding $i\colon D\into I$ be
defining an object of $\ccat{D}(I)$.
(Denote the object simply by $D$.)
Then we want to prove that the category $\Dis(I)_{D/}$ has
contractible classifying space.

In other words, we want to prove that the category
\[
\laxcol_{E\in\Dis(I)}\Fibre\left[\Emb(D,E)\longto\Emb(D,I)\right],
\]
fibre taken over $i$, has contractible classifying space, which is
the colimit of the same diagram (rather than the lax colimit in the
$2$-category of categories), and thus equivalent to
\[
\Fibre\left[\colim_{E\in\Dis(I)}\Emb(D,E)\longto\Emb(D,I)\right].
\]
In this last step, we have used the standard equivalence between the
category of spaces over $\Emb(D,I)$, and the category of local
systems of spaces on $\Emb(D,I)$ (to be elaborated on in Remark
below for completeness).

In fact, we can prove that the map
$\colim_{E\in\Dis(I)}\Emb(D,E)\to\Emb(D,I)$ is an
equivalence as follows.

Let $D_0$ be the union of the components of $D$ which are open
intervals.
Choose a homeomorphism $D_0\equivwith S\times\R^1$ for a finite set
$S$.
In particular, we have picked a point in each component of $D_0$,
corresponding to the origin in $\R^1$, together with the germ of a
chart at the chosen points.
Then, given an embedding $D_0\into U:=E-\boundary E$, where
$E\in\Dis(I)$, restriction of it to the germs of charts at the chosen
points gives us an injection $S\into U$ together with germs of charts
in $U$ at the image of $S$.
This defines a homotopy equivalence of
$\Emb(D_0,U)$($\equivwith\Emb(D,E)$) with the space of germs of
charts around distinct points in $U$, labeled by $S$.

Furthermore, this space is fibred over the configuration
space $\Conf(S,U):=\Emb(S,U)/\Aut(S)$, with fibres equivalent to
$\mathrm{Germ}_0(\R^1)\wreath\Aut(S)$, where $\mathrm{Germ}_0(\R^1)$
is from \cite[Notation 5.2.1.9]{higher-alg}.

It follows that the task has been reduced to proving that the map
\[
\colim_{E\in\Dis(I)}\Conf(\pi_0(D_0),E-\boundary
E)\longto\Conf(\pi_0(D_0),I-\boundary I)
\]
is an equivalence.

However, for any finite set $S$, the cover determined by the functor
$E\mapsto\Conf(S,E-\boundary E)\sub\Conf(S,I-\boundary I)$ satisfies
the hypothesis for the generalised Seifert--van Kampen theorem
\cite{higher-alg}.
\end{proof}

The following is a side remark on a result we have used.
\begin{remark}\label{rem:homotopy-theory}
In the proof, we have used an equivalence between the category of
spaces over a space $X$, and the category of local systems of spaces
on $X$.
This is simply a special case of a standard fact on Grothendieck
fibrations.
Indeed, spaces are also known as groupoids, so ``a space over $X$''
is a rephrasing of ``a category fibred over $X$ in groupoids''
(both consist of identical data with equivalent required properties),
and every functor over $X$ preserves Cartesian maps.
(Note that every functor with target a groupoid is a fibration, and a
map in the source is Cartesian if and only if it is an equivalence.)
\end{remark}

Let us denote by $\ccat{D}(I)_{\boundary I}$ the full subcategory of
$\ccat{D}(I)$ consisting of objects each of which contains (as a
submanifold of $I$) both points of $\boundary I$.
Note that the inclusion $\ccat{D}(I)_{\boundary I}\to\ccat{D}(I)$ is
cofinal.

\begin{lemma}\label{lem:equiv-to-simplices}
The functor $\ccat{D}(I)_{\boundary I}\to\Simp^\op$ given by
$U\mapsto\pi_0(I-U)$ (with the order inherited from the order in $I$)
is an equivalence.
\end{lemma}
\begin{proof}
A simple verification.
\end{proof}

Let us denote by $\Dis(I)_{\boundary I}$ the similar full subposet of
$\Dis(I)$.
In other words, $\Dis(I)_{\boundary I}:=\ccat{D}(I)_{\boundary
  I}\times_{\ccat{D}(I)}\Dis(I)$.

\begin{corollary}\label{cor:cofinal-rel-boundary}
The canonical functor $\Dis(I)_{\boundary I}\to\ccat{D}(I)_{\boundary
  I}$ is cofinal.
\end{corollary}
\begin{proof}
This follows from Lemma \ref{lem:disj-cofinal-d}.
Indeed, for every $U\in\ccat{D}(I)_{\boundary I}$, the induced functor
$(\Dis(I)_{\boundary I})_{U/}\to\Dis(I)_{U/}$ is an equivalence.
\end{proof}

\begin{lemma}\label{lem:cofinal-sub}
The inclusion functor $\Dis(I)_{\boundary I}\into\Dis(I)$ is cofinal.
\end{lemma}
\begin{proof}
We apply Joyal's generalisation of Quillen's theorem A \cite{topos}.

Namely, we would like to prove for any given $U\in\Dis(I)$ that the
category $(\Dis(I)_{\boundary I})_{U/}$ has contractible classifying
space.

Let $(\Dis(I)_{\boundary I})_U$ denote the full subposet of
$(\Dis(I)_{\boundary I})_{U/}$ consisting of maps $U\into V$ in
$\Dis(I)$ such that $V$ is of the form $U\disj D$ for a collar $D$ of
$\boundary I\intersect(I-U)$ in $I$ which is disjoint with $U$.
Then the inclusion $(\Dis(I)_{\boundary I})_U\into(\Dis(I)_{\boundary
  I})_{U/}$ has a right adjoint, so induces a homotopy equivalence
on the classifying spaces.
However, the classifying space of $(\Dis(I)_{\boundary I})_U$ is
contractible since it has maximum.
\end{proof}

We conclude from Lemma \ref{lem:cofinal-sub}, Corollary
\ref{cor:cofinal-rel-boundary} and Lemma \ref{lem:equiv-to-simplices},
that we have obtained a pair of cofinal functors
\[
\Dis(I)\longfrom\Dis(I)_{\boundary I}\longto\Simp^\op.
\]

\medskip

Now let $A$ be a constructible prealgebra on $I$, and
let $B$, $K$, $L$ be the algebras and modules obtained from $A$.
Then it follows from the above constructions that the functor
$A\resto{\Dis(I)_{\boundary I}}$ is equivalent to the restriction to
$\Dis(I)_{\boundary I}$ of the simplicial bar construction
$B_\dot(K,B,L)\colon\Simp^\op\to\cat{A}$.

In particular, we obtain a canonical equivalences
\[
\colim_{\Dis(I)}A\xlongleftarrow{\equiv}\colim_{\Dis(I)_{\boundary
    I}}A\xlongrightarrow{\equiv}\colim_{\Simp^\op}B_\dot(K,B,L)=K\tensor_BL.
\]

Let us denote the points of $\boundary I$ by $x_+$ and $x_-$.
Recall that the underlying object of the algebra ``$B$'' is identified
with $A(I-\boundary I)$.
Let us assume that our conventions identify $A(I-x_-)$ with the right
$B$-module ``$K$'', and $A(I-x_+)$ with the left $B$-module
``$L$''.

In this way, objects $A(I-\boundary I)$ and $A(I-x_\pm)$ get a
canonical structure of an algebra and its left/right module
respectively.

\begin{proposition}\label{prop:homology-over-interval}
Let $A$ be a constructible prealgebra on $I$, and assume that, as a
functor, $A$ preserves filtered colimits.
Then $A$ is a factorisation algebra if and only if the canonical map
\[
A(I-x_{-})\tensor_{A(I-\boundary I)}A(I-x_{+})\longto A(I)
\]
is an equivalence.
Here, the tensor product is with respect to the canonical structures.
\end{proposition}

\subsection{Excision}
Following Francis, we shall introduce the notion of excision and
review after him, the relation between excision and other descent
properties.

\begin{definition}
Let $M$ be a manifold (without boundary).
Then we say that a map $p\colon M\to I$ is \kore{constructible} if
$p\colon p^{-1}(I-\boundary I)\to I-\boundary I$ is locally trivial,
i.e., is the projection of a fibre bundle.
\end{definition}

Let $N:=p^{-1}(t)\sub M$ be the fibre of a point $t\in I-\boundary I$.
Then $N$ is a smooth submanifold of $M$ of codimension $1$, and its
normal bundle in $M$ is a trivial line bundle.

We can write $I=I_0\union_{t}I_1$ where $I_0$ is the points
of $I$ below or equal to $t$, and $I_1$ is the points of $I$ above or
equal to $t$.
Accordingly, the total space $M$ can be written in the glued form
$M_0\union_NM_1$ where $M_i=p^{-1}I_i$.

Conversely, if $M$ is given a decomposition $M_0\union_NM_1$
with $N$ a submanifold of codimension $1$ with trivial normal bundle,
then we have a constructible map $p\colon M\to I$ for an interval $I$
so the decomposition of $M$ can be reconstructed as above from $p$.
It suffices to choose a trivialisation of the normal bundle of $N$ (in
the desired orientation) to construct $p$.

The excision property is concerned with what happens to the value
associated by a prealgebra when $M$ is constructed by gluing as above.

Let $p\colon M\to I$ be constructible, then for a locally constant
algebra $A$ on $M$, it is immediate from the isotopy invariance that
the prealgebra $p_*A$ on $I$ is constructible.

\begin{definition}
Let $A$ be a locally constant prealgebra on $M$.
We say that $A$ satisfies \kore{excision with respect to} a
constructible map $p\colon M\to I$ if the prealgebra $p_*A$ on $I$ is
a constructible factorisation algebra on $I$.

We say that $A$ satisfies \kore{excision} if for every $U\sub M$
equipped with a constructible map $p\colon U\to I$, $A\resto{U}$
satisfies excision with respect to $p$.
\end{definition}

The following is a formulation in our setting of a fundamental
fact discovered by Francis, with proof also following his ideas.

\begin{theorem}[Francis]\label{thm:excision}
A locally constant prealgebra on a manifold is a factorisation
algebra, namely its underlying functor is a left Kan extension from
disjoint unions of disks \cite[Section 2]{descent},
if and only if it satisfies excision.
\end{theorem}

In order to prove this, we recall the following definition and a
theorem.
\begin{definition}[{\cite[Definition 2.5]{descent}}]
Let $\cat{C}$ be a category and let
$\chi\colon\cat{C}\to\Open(M)$ be a functor.
For $i\in\cat{C}$, denote $\chi(i)$ also by $U_i$ within this
definition.
We shall call this data a \kore{factorising cover} which is \kore{nice
  in Lurie's sense}, or briefly, \kore{factorising l-nice cover}, of
$M$ if for any non-empty finite subset $x\sub M$, the full subcategory
$\cat{C}_x:=\{i\in\cat{C}\:|\:x\subset U_i\}$ of $\cat{C}$ has
contractible classifying space.
\end{definition}

\begin{theorem}[{\cite[Theorem 2.11]{descent}}]
\label{thm:factorising-seifert-vk}
Let $A$ be a locally constant algebra on $M$ (in a symmetric monoidal
category $\cat{A}$ satisfying our assumption stated in
Section \ref{sec:manifold-algebra}).
Then for any factorising l-nice cover determined by $\chi\colon
\cat{C}\to\Open(M)$, the map
$A(M)\from\colim_\cat{C}A\chi$ is an equivalence.
\end{theorem}

\begin{proof}[Proof of Theorem \ref{thm:excision}]
Let $A$ be a locally constant factorisation algebra on a manifold $M$,
and let us prove that $A$ satisfies excision.
For this purpose, let $U$ be an open submanifold equipped with a
constructible map $p\colon U\to I$.
We need to prove that the constructible algebra $p_*(A\resto{U})$ on
$I$ is a left Kan extension from disjoint unions of subintervals.

For notational convenience, denote $A\resto{U}$, which satisfies the
same assumption as $A$ does, just by $A$.

We want to prove that for every open submanifold $V\sub I$, the value
$(p_*A)(V)$ is equivalent to $\colim_{D\in\Dis(V)}(p_*A)(D)$ by the
canonical map.
Namely, $A(p^{-1}V)=\colim_{D\in\Dis(V)}A(p^{-1}D)$.

Since the objects of $\Dis(V)$ form a factorising l-nice cover of $V$,
the functor $p^{-1}\colon\Dis(I)\to\Open(p^{-1}V)$ determines a
factorising l-nice cover of $p^{-1}V$.
Therefore, the result follows from Theorem
\ref{thm:factorising-seifert-vk}.

The converse now follows as follows.

Firstly, if a locally constant prealgebra
satisfies excision, then this prealgebra as well as the factorisation
algebra obtained from it as a left Kan
extension of its restriction to disjoint unions of disks, both satisfy
excision.
Then since every open submanifold of $M$ (or rather its compact
closure) has a handle body decomposition, two prealgebras coincide as
soon as they coincide on open submanifolds diffeomorphic to a disk or
$D^i\times\boundaryc{D}^j$.

However, the two prealgebras do coincide on disks by construction, and
then also on $D^i\times\boundaryc{D}^j$ by inductively (on $j$)
applying excision.
\end{proof}

\begin{remark}\label{rem:functoriality-of-constructible-push}
The construction similar to that in \cite[Section 3.0]{descent} of the
functoriality for the push-forward operation on the groupoid of
locally trivial maps
shows that the push-forward $p_*A$ of a locally constant algebra is
naturally functorial in $p$ (on the groupoid of constructible maps).
\end{remark}

\section{Compactly supported factorisation homology}
\label{sec:compactly-supported-homology}
\setcounter{subsection}{-1}
\setcounter{equation}{-1}

Let $M$ be a manifold without boundary, and let $A$ be a locally
constant factorisation algebra on $M$.
Recall that $A$ determines a functor $U\mapsto A(U):=\int_UA$ by
factorisation homology.
When $A$ is equipped with an \emph{augmentation}, namely, an algebra
map $A\to\unity$, we shall define \emph{compactly supported}
factorisation homology $\int^c_UA$ with coefficients in $A$, and shall
make it into a symmetric monoidal \emph{contravariant} functor of $U$.

\begin{remark}
Although this is implicit in our notation $\int^c_UA$, compactly
supported homology will be defined as \emph{dependent} on the
compactification of $U$ which comes with $U$ in our convention (see
Section \ref{sec:manifold-algebra}).
\end{remark}

For the purpose of definitions in this section, by an \kore{interval},
we mean an oriented smooth compact connected manifold of dimension
$1$, with exactly one incoming boundary point and exactly one outgoing
point with respect to the orientation.

\subsection{Description of the object}
\label{sec:object}
Let us first define, for each open submanifold $U\sub M$, the
compactly supported homology object $\int^c_UA$.

Let $I$ be an interval with incoming point $s$ and outgoing point $t$.
We choose our conventions so a constructible algebra on $I$ with
stratification specified by its boundary, is given by an associative
algebra $B$ to be on the interior, a \emph{right} $B$-module $K$ to be
on the point $s$, a \emph{left} $B$-module $L$ to be on the point $t$,
and (right or left) $B$-module maps $B\to K$ and $B\to L$.
Recall then the factorisation homology over $I$ is $K\tensor_BL$
(Proposition \ref{prop:homology-over-interval}).

Let $\close{U}$ denote the specified compact closure of $U$.
For the construction, we choose a (constructible) $C^\infty$-map
$p\colon\close{U}\to I$ which is locally trivial over $I-\{s\}$, and
such that $p^{-1}(t)=\boundaryc{U}$, so $p$ restricts to a
constructible map $U\,\to\,I-\{t\}$.
By a constructible (resp.~locally trivial) map, we always mean a
smooth map which is constructible (resp.~locally trivial) in the
category of smooth manifolds and $C^\infty$ maps.

With these data specified, we have an algebra $p_*A$ on $I$.
Let $j$ denote the inclusion $I-\{t\}\,\into\,I$.
Then the augmentation of $A$, and hence of $p_*A$, allows one to
extend $j^*p_*A$ (the restriction of $p_*A$ along $j$) along $j$ by
putting the module $\unity$ on $t$.
Let us denote this augmented algebra on $I$ by $j_!\,j^*p_*A$.

We then define the \kore{compactly supported} factorisation homology
over $U$ to be
\[
\int^c_UA\::=\:\int_I\,j_!\,j^*p_*A.
\]
In other words, it is $A(U)\tensor_{A(\boundaryc{U})}\unity$, where
$A(\boundaryc{U})$ is the associative algebra we find on the interior
of $I$ from $p_*A$.
Since the pushforward of the augmented algebra $A$ is functorial in
$p$ on the groupoid of constructible maps by
Remark \ref{rem:functoriality-of-constructible-push}, $\int^c_UA$ is
unambiguously defined by $A$ and $U$.

\subsection{Functoriality}
\label{sec:functoriality}
Next we would like to make the association $U\mapsto\int^c_UA$
contravariantly functorial in $U$.
Let us start with some preparation.
Let $N_\dot$ denote the nerve functor from the category of
categories to the category of simplicial spaces.
Then, since $N_\dot$ is fully faithful by \cite[Proposition
A.7.10]{higher-alg}, a functor $\Open(M)^\op\to\cat{A}$ is the same as
a map $N_\dot\Open(M)^\op\to N_\dot\cat{A}$ of simplicial spaces.

Let $\Simp$ denote the category of combinatorial simplices.
Its objects are non-empty totally ordered finite sets, and maps are
order preserving maps.
Then the data of a simplicial space is equivalent to the
category fibred over $\Simp$ in groupoids, obtained by taking the lax
colimit over $\Simp^\op$.
Moreover, the desired simplicial map is then equivalent to a functor
$\laxcol_{\Simp^\op}N_\dot\Open(M)^\op\to\laxcol_{\Simp^\op}N_\dot\cat{A}$
over $\Simp$.
Indeed, every map in a category fibred in groupoids is Cartesian
over the base.

Let us denote $\laxcol_{\Simp^\op}N_\dot\Open(M)^\op$ by $\cat{X}$.
Again, it suffices to construct a map $N_\dot\cat{X}\to
N_\dot\laxcol_{\Simp^\op}N_\dot\cat{A}$ of simplicial spaces over
$N_\dot\Simp$.
In order to construct this, we replace $N_\dot\cat{X}$ by a
simplicial space $X_\dot$ equivalent to it.
Namely, we shall construct $X_\dot$, an equivalence $X_\dot\equivto
N_\dot\cat{X}$, and a map $X_\dot\to
N_\dot\laxcol_{\Simp^\op}N_\dot\cat{A}$.

\bigskip

Let us first describe $X_0$.
We define it as the colimit (coproduct) over
$N_0\cat{X}=\colim_{[k]\in N_0\Simp}N_k\Open(M)$ of certain spaces,
each of which will turn out to be contractible.
Namely, given an integer $k$ and a $k$-nerve $U\colon
U_0\into\cdots\into U_k$ of $N_k\Open(M)$, we associate to it the
natural space $X_U$ formed by pairs $(p,I')$, where
\begin{enumerate}
\setcounter{enumi}{-1}
\item\label{item:choice-of-constructible-map}
  $p=(p_i)_{i\in[k]}$, where $p_i\colon\close{U}_i\to I_i$ is a ($C^\infty$)
  constructible map (in the $C^\infty$ sense of the word) to an
  interval as before, and
\item\label{item:subinterval}
  $I'=(I'_i)_{i\in[k]}$, where $I'_i$ is a subinterval of $I_i$
\end{enumerate}
which are required to satisfy the following condition.
Namely, give $I_i$ a total order which recovers its topology, and for
which $s_i<t_i$ for the incoming end point $s_i$ and the outgoing end
point $t_i$, so $I_i=[s_i,t_i]$.
Write $I'_i=[s'_i,t'_i]$ in this order.
Then the required condition will be that
\[
p_i^{-1}[s_i,t'_i]\:\sub\:p_j^{-1}[s_j,s'_j]
\]
in $U_j$, whenever $i\lneq j$.

We define $X_0$ to be the coproduct of $X_U$ over all $k$ and $U$.

\begin{claim}
The map $X_0\to N_0\cat{X}$ forgetting $p$ and $I'$ is an equivalence.
\end{claim} 
\begin{proof}
It suffices to prove that, for every $k$, $U\in N_k\Open(M)$ and
\eqref{item:choice-of-constructible-map} above, choices for
\eqref{item:subinterval} form a contractible space.

Note that the required condition implies
$p_i^{-1}(s_i)\sub p_j^{-1}[s_j,s'_j
)$ for $i\lneq j$.
Now the space of $s'_k$ satisfying $p_i^{-1}(s_i)\sub
p_k^{-1}[s_k,s'_k
)$ for every $i\lneq k$, is contractible.
Moreover, once $s'_j$ is chosen so we have
\[
p_i^{-1}(s_i)\sub p_j^{-1}[s_j,s'_j
)
\]
for all $i\lneq j$, then the space of $s'_{j-1}$ satisfying
$p_i^{-1}(s_i)\sub p_{j-1}^{-1}[s_{j-1},s'_{j-1}
)$ for every $i\lneq j-1$, and $p_{j-1}^{-1}[s_{j-1},s'_{j-1}]\sub
  p_j^{-1}[s_j,s'_j
)$, is contractible.
Moreover, once $s'_j$'s are all chosen, so these conditions are
satisfied, then the space of $t'_j$'s satisfying the required
condition is contractible.
This completes the proof.
\end{proof}

Before describing $X_\kappa$ for $\kappa\ge 1$, let us construct a map
$X_0\to N_0\laxcol_{\Simp^\op}N_\dot\cat{A}$ over $N_0\Simp$, to be
the simplicial level $0$ of the desired map $X_\dot\to
N_\dot\laxcol_{\Simp^\op}N_\dot\cat{A}$.
For an integer $k\ge 0$, let
\[
X_{0k}\::=\:\coprod_{U\in
  N_k\Open(M)}X_U\:\big(=\:X_0\times_{N_0\cat{X}}N_k\Open(M)\big),
\]
so $X_0=\colim_{[k]\in N_0\Simp}X_{0k}$.
We construct maps $X_{0k}\to N_k\cat{A}$ for all $k$, and define
the desired map as the colimit (coproduct) of these maps over $[k]\in
N_0\Simp$.
The map $X_{0k}\to N_k\cat{A}$ will be defined by induction on $k$ as
follows.

We first define a map $X_{00}\to N_0\cat{A}$.
Thus, suppose given $U\in\Open(M)$ and a pair consisting of
$p\colon\close{U}\to I$ and $I'\sub I$ as above.
Then define $p'\colon\close{U}\to I'$ as the composite of $p$
with the map $I\onto I'$ which collapses each of $[s,s']$ and
$[t',t]$.
Let $j'$ denote the inclusion $I'-\{t'\}\,\into\,I'$.
Then we associate to the point $(U,p,I')$ of $X_{00}$ the point
$\int_{I'}j'_!\,j'^*p'_*A$ of $N_0\cat{A}$.
This is functorial on the groupoid $X_{00}$ by the functoriality of
the push-forward construction on the space of constructible maps
(Remark \ref{rem:functoriality-of-constructible-push}), and the value
is canonically equivalent to the compactly supported homology
$\int^c_UA$.
Note that the value for $(U,p,I')$ can also be written as
$\int_{I'}j'_!\,j'^*(p'\resto{\close{U}{}'})_*A$, where
$\close{U}{}':=p^{-1}[s,t']\sub\close{U}$.

Next, we construct a map $X_{01}\to N_1\cat{A}$.
Thus, consider a point of $X_{01}$ given by a $1$-simplex $U\colon
U_0\into U_1$ of $N_\dot\Open(M)$, and $(p,I')\in X_U$.
Then we would like to construct a $1$-simplex of $N_\dot\cat{A}$ to be
associated to it.

In order to do this, define $p'_i$ to be the composite of $p_i$ with
the map $I_i\onto I'_i$ collapsing each of $[s_i,s'_i]$ and
$[t'_i,t_i]$, and define $\close{U}'_i:=p_i^{-1}[s_i,t'_i]$.
Then
$\close{U}_1=\close{U}'_0\union_{p_0^{-1}(t'_0)}(\close{U}_1-U'_0)$,
and the maps $p'_0$ and $p'_1\resto{\close{U}_1-U'_0}$ glue together
to define a map
\[
p\colon\close{U}_1\longto I'_0\union I'_1=:I,
\]
where the intervals are glued by the relation $t'_0=s'_1$.

Using this, we have the following.
Let $j_0$ denote the inclusion $I'_0-\{t'_0\}\,\into\,I'_0$, $j_1$
denote the inclusion $I-\{t'_1\}\,\into\,I$, and $q$ denote the map
$I\onto I'_0$ collapsing $I'_1$.
Then the induced augmentation of
$B:=q_{*\,}j_{1!}\,j_1^*\,p_*A$ induces a map
$B\,\to\,j_{0!}\,j_0^*B=j_{0!}\,j_0^*\,p'_{0*}A$.

Integrating this over $I'_0$, we obtain a map
\[
\int_{I'_1}\,j_{1!}\,j_1^*\,p'_{1*}A\:=\:\int_I\,j_{1!}\,j_1^*\,p_*A\,\longto\,\int_{I'_0}\,j_{0!}\,j_0^*\,p'_{0*}A.
\]
The source and the target here are models of the compactly supported
homology associated to the composites
$\close{U}_i\xrightarrow{p_i}I_i\onto[s_i,t'_i]$
for $i=1,0$ respectively, where the latter map collapses the
subinterval $[t'_i,t_i]$ of $I_i$.
The $1$-simplex in $\cat{A}$ we would like to associate is this map
with $0$-faces given by its source and target.
This $1$-simplex (as an object of the groupoid $N_1\cat{A}$) is
functorial on the groupoid $X_{01}$, since the push-forward of the
augmentation map (together with the push-forward of the algebra)
depends functorially on the space of constructible maps by Remark
\ref{rem:functoriality-of-constructible-push}.
Therefore, we have constructed a map $X_{01}\to N_1\cat{A}$ as
desired.

Next, we construct a map $X_{02}\to N_2\cat{A}$.
Thus, suppose given a $2$-simplex $U\colon U_0\into U_1\into U_2$ of
$N_\dot\Open(M)$, and $(p,I')\in X_U$.
In order to construct a $2$-simplex of $N_\dot\cat{A}$ from these
data, note that the previous constructions applied to $0$- and
$1$-faces of $(U,p,I')$ (by which we mean the faces of the nerve $U$
equipped with the restrictions of $(p,I')$ there) give what could be
the boundary of a $2$-simplex $[2]\to\cat{A}$.
In order to fill inside of this by a $2$-isomorphism to actually get a
$2$-simplex, define $p'_i\colon\close{U}_i\to I'_i$ and
$\close{U}'_i\sub\close{U}_i$ as before.
Then the maps $p'_0\colon\close{U}'_0\to I'_0$,
$p'_1\colon\close{U}'_1-U'_0\,\to\,I'_1$,
$p'_2\colon\close{U}_2-U'_1\,\to\,I'_2$ glue together to define a map
\[
p\colon\close{U}_2\longto I'_0\union I'_1\union I'_2=:I,
\]
where the union is taken under the relations $t'_i=s'_{i+1}$.
Let $q_{ij}\colon I\onto I_{ij}:=I'_i\union I'_j$ be the map
collapsing the other interval, $q_j\colon I_{ij}\to I'_j$ be the map
collapsing $I'_i$, $j_2\colon I-\{t'_2\}\,\into\,I$ be the inclusion,
and $B:=q_{01*\,}j_{2!}\,j_2^*\,p_*A$ on $I_{01}$.
Then the $2$-isomorphism we would like to find is between
$\int_{I'_0}$ of the augmentation map
$q_{0*}B\:\to\:j_{0!}\,j_0^*\,q_{0*}B$, and $\int_{I'_0}$ of the
composite
\[
q_{0*}B\xlongrightarrow{q_{0*}\varepsilon}q_{0*\,}j_{1!}\,j_1^*B\xlongrightarrow{\varepsilon}j_{0!}\,j_0^*\,q_{0*\,}j_{1!}\,j_1^*B=j_{0!}\,j_0^*\,q_{0*}B
\]
of augmentation maps.
We find an isomorphism between these, induced from the data of
multiplicativity/monoidality of the augmentation map.
Moreover, the $2$-simplex we have thus constructed again depends
functorially on $X_{02}$, since the push-forward of all data
associated with augmentation maps are functorial on the space of
constructible maps by Remark
\ref{rem:functoriality-of-constructible-push}.
Therefore, we have constructed a map $X_{02} \to N_2\cat{A}$.

Inductively, let a point $(U,p,I')\in X_{0k}$ be given by $U\in
N_k\Open(M)$ and $(p,I')\in X_U$.
Then the previous steps applied to the boundary faces of $(U,p,I')$
give a diagram of the shape of $\boundary\simp^k$ in $\cat{A}$.
We would like to fill this by a $k$-isomorphism to obtain a
$k$-simplex $[k]\to\cat{A}$, which we can then associate to
$(U,p,I')$.

In the case $k=3$, the previous step of the induction implies that the
$1$-faces of this $\boundary\simp^3$-shaped diagram are the
$1$-isomorphisms induced from the suitable instances of the
augmentation map $\epsilon$ of $A$, and the $2$-faces are the
$2$-isomorphisms induced from the appropriate instances of the
$2$-isomorphism of multiplicativity of the augmentation map
$\epsilon$.
We get a $3$-simplex in $\cat{A}$ by filling the inside of this
$\boundary\simp^3$ shape by the
$3$-isomorphism in $\cat{A}$ induced from the $3$-isomorphism of
coherence of the multiplicativity of the augmentation map.
We associate this $3$-simplex of $N_\dot\cat{A}$ to $(U,p,I')\in
X_{03}$.

In the case $k=4$, the faces of dimension up to $2$ of the
$\boundary\simp^4$-shaped diagram are similar to those in the previous
case, and the $3$-faces are the instances of $3$-simplices constructed
in the previous step from the $3$-isomorphisms of coherence of the
multiplicativity of the augmentation map.
We get a $4$-simplex by filling inside this $\boundary\simp^4$ shape
by the $4$-isomorphism induced from the $4$-isomorphism here of the
next level coherence of the multiplicativity of the augmentation map.

For a general $k$, we may assume by induction, that the previous steps
are done by similarly taking the isomorphisms of dimension up to
$k-1$, all induced from the appropriate instances of the coherence
isomorphisms.
In particular, the simplices we obtain from the boundary faces of
$(U,p,I')$, are made up of the isomorphisms induced from the coherence
isomorhisms (up to dimension $k-1$) of the multiplicativity of the
augmentation map, and the shape of $\boundary\simp^k$ in $\cat{A}$ is
therefore made up of these simplices.
Then, as promised by the inductive hypothesis, we associate to
$(U,p,I')$ the $k$-simplex $[k]\to\cat{A}$ obtained by filling inside this
$\boundary\simp^k$ shape by the $k$-isomorphism in $\cat{A}$ induced
from the $k$-isomorphism here of the next level coherence of the
multiplicativity of the augmentation map of $A$.
This construction of a $k$-simplex is again functorial on $X_{0k}$.
Therefore, we have constructed a map $X_{0k}
\to N_k\cat{A}$ in the way required by the next inductive hypothesis.
This completes the induction on $k$, and therefore the construction of
a map $X_0\to N_0\laxcol_{\Simp^\op}N_\dot\cat{A}$ over $N_0\Simp$.

\bigskip

Let us next construct the space $X_1$.
We define it as the coproduct over $N_1\cat{X}$ of certain spaces,
each of which will turn out contractible.
First, note that a point of $N_1\cat{X}$ is specified by integers
$k,\ell\ge 0$, a map $\phi\colon[k]\to[\ell]$ in $\Simp$, and $U\in
N_\ell\Open(M)$.
Given a point of $N_1\cat{X}$ specified by these data, we let the
space naturally formed by the following be the component of
$X_1$ lying over it.
\begin{enumerate}
\setcounter{enumi}{-1}
\item A point $(p,I')\in X_U$.
\item For every $i\in[k]$, a subinterval $J_i=[u_i,v_i]$ of $I'_{\phi
    i}$ such that $v_i\le u_{i+1}$ whenever $\phi i=\phi(i+1)$.
\item A map in the fundamental groupoid of the space
  \begin{equation}\label{eq:space-of-sequences}
    \Map_{/[\ell]}(\phi,I'):=\prod_{j\in[\ell]}\Map_{\le}(\phi^{-1}j,I'_j)
  \end{equation}
  (where $\Map_{\le}$ stands for the space of order preserving maps),
  from $v=(v_i)_{i\in[k]}$ to $\phi^*t'=(t'_{\phi i})_{i\in[k]}$,
  where we are identifying a point of $\Map_{/[\ell]}(\phi,I')$ in
  general with an increasing sequence $x=(x_i)_{i\in[k]}$ of points in
  $\Union_{j\in[\ell]}I'_j$ such that $x_i\in I'_{\phi i}$.
\end{enumerate}
Note, as we have claimed, that the space is contractible, so the
projection $X_1\to N_1\cat{X}$ is an equivalence.

\bigskip

More generally, for an integer $\kappa\ge 2$, we let $X_\kappa$ be
the coproduct over $N_\kappa\cat{X}$ of the following (again,
contractible) spaces.
Namely, let a point of $N_\kappa\cat{X}$ be specified by a
$\kappa$-nerve
\begin{equation}\label{eq:nerve-of-simplices}
\phi\colon[k_\kappa]\xlongrightarrow{\phi_\kappa}\cdots\xlongrightarrow{\phi_1}[k_0]
\end{equation}
and $U\in N_{k_0}\Open(M)$, where we consider
$[\kappa]$ as $\{\kappa\to\cdots\to 0\}$ for notational convenience.
Then we take the natural space formed by the following, for the
component of $X_\kappa$ to lie over the specified point of
$N_\kappa\cat{X}$.
\begin{enumerate}
\setcounter{enumi}{-1}
\item A point $(p,J^0)\in X_U$.
\item For every $\iota\in[\kappa]-\{0\}$, a family
  $J^\iota=(J^\iota_i)_{i\in[k_\iota]}$ of subintervals
  $J^\iota_i=[u^\iota_i,v^\iota_i]$ of $J^{\iota-1}_{\phi_\iota i}$,
  satisfying $v^\iota_i\le u^\iota_{i+1}$ if $\phi_\iota
  i=\phi_\iota(i+1)$.
\item For every $\iota\in[\kappa]-\{0\}$, a map in the
  fundamental groupoid of the space
  $\Map_{/[k_{\iota-1}]}(\phi_\iota,J^{\iota-1})$ (see
  (\ref{eq:space-of-sequences}) above),
  from $v^\iota$ to $\phi_\iota^*(v^{\iota-1})$.
\end{enumerate}
Note that this space is contractible as claimed, so the projection
$X_\kappa\to N_\kappa\cat{X}$ is an equivalence.

\bigskip

It follows that there uniquely exists a pair consisting of a
simplicial structure on $X_\dot$ and a simplicial structure on the
level-wise map $X_\dot\to N_\dot\cat{X}$ given by the projections.
We will want to use a more concrete description of this.

In order to get a desired description of the simplicial structure, for
$\iota\in[\kappa]$, let $\psi_\iota$ denote the composite
$\phi_1\cdots\phi_\iota\colon[k_\iota]\to[k_0]$.
Then note that the pair $(\psi_\iota^*p,J^\iota)$ determines a point
of $X_{\psi_\iota^*U}$.
Now suppose given a map $f\colon[\lambda]\to[\kappa]$ in $\Simp$.
We would like to define a map $f^*\colon X_\kappa\to X_\lambda$.
In order to do this, suppose given a point of $X_\kappa$ over the
point $(\phi,U)\in N_\kappa\cat{X}$, specified as above.
Then we let $f^*$ associate to this point a point of $X_\lambda$ lying
over $f^*(\phi,U)=\big(f^*\phi,(f^*\phi)_0^*U\big)\in
N_\lambda\cat{X}$, specified by
\begin{enumerate}
\setcounter{enumi}{-1}
\item the point $(\psi_{f(0)}^*p,J^{f(0)})\in
  X_{\psi_{f(0)}^*U}=X_{(f^*\phi)_0^*U}$,
\item $f^*J=(J^{f(\iota)})_{\iota\in[\lambda]-\{0\}}$, and
\item for every $\iota\in[\lambda]-\{0\}$, the map
  $v^{f(\iota)}\to(f^*\phi)_\iota^*(v^{f(\iota-1)})$
  in the fundamental groupoid of
  $\Map_{/[k_{f(\iota-1)}]}\left((f^*\phi)_\iota,J^{f(\iota-1)}\right)$,
  obtained by composing the paths determining the point of $X_\kappa$.
\end{enumerate}
This extends to a map $X_\kappa\to X_\lambda$ by functoriality, and we
define $f^*$ as this map.
By the associativity of composition of maps in the fundamental
groupoids, this construction defines a functoriality on $\Simp$, and
we have thus described a simplicial structure of $X_\dot$, lying over
the simplicial structure of $N_\dot\cat{X}$.

\bigskip

Let us finally extend the map $X_0\to
N_0\laxcol_{\Simp^\op}N_\dot\cat{A}$ over $N_0\Simp$ we have
constructed, to a full simplicial map over $N_\dot\Simp$.

For this purpose, suppose given the following partial data towards a
$1$-simplex of $X_\dot$.
The data we consider are a $1$-simplex $(\phi,U)\in N_1\cat{X}$
($\phi\colon[k_1]\to[k_0]$), which corresponds to a component of
$X_1$, a family $p=(p_j)_{j\in[k_0]}$ of constructible maps
$p_j\colon\close{U}_j\to I_j$ as before, $J^0=(J^0_j)_{j\in[k_0]}$
($J^0_j=[u^0_j,v^0_j]\sub I_j$) such that $(p,J^0)\in X_U$.
Then we construct a map $\Map_{/[k_0]}(\phi,J^0)\to N_{k_1}\cat{A}$ as
follows.
Namely, let a point of $\Map_{/[k_0]}(\phi,J^0)$ be represented by an
increasing sequence $x=(x_i)_{i\in[k_1]}$ of $\Union_{[k_0]}J^0$.
Then we construct a $k_1$-simplex $[k_1]\to\cat{A}$ from this data
by making the following modifications to the construction of the map
$X_{0k_1}\to N_{k_1}\cat{A}$ we have done before.

To be precise on the comparison with the previous construction,
what we shall construct is a $k_1$-simplex of $N_\dot\cat{A}$ which
specialises to the $k_1$-simplex associated to
$(\phi^*U,\phi^*p,J^1)\in X_{0k_1}$ (denote it by
$\int^c_{\phi^*U,\phi^*p,J^1}A$) if $x=v^1$ for a family
$J^1=(J^1_i)_{i\in[k]}$ of subintervals $J^1_i=[u^1_i,v^1_i]\sub
J^0_{\phi i}$ satisfying the conditions we have described in the
definition of the space $X_1$, so $(\phi^*p,J^1)\in X_{\phi^*U}$.
The modification will be made to the construction of
$\int^c_{\phi^*U,\phi^*p,J^1}A$ from the data
$(\phi^*U,\phi^*p,J^1)\in X_{0k_1}$.
Its description as follows.

Firstly, by denoting $p_{\phi i}^{-1}[s_{\phi i},v^1_i]$ by
$\close{V}_i$, the construction of $\int^c_{\phi^*U,\phi^*p,J^1}A$
used the constructible map
\[
\Union_{i\in[k_1]}q_i\resto{\close{V}_i-V_{i-1}}\colon\,\close{V}_{k_1}\longto\Union_{i\in[k_1]}J^1_i,
\]
where $q_i\colon\close{U}_{\phi i}\xrightarrow{p_{\phi i}}I_{\phi
  i}\onto J^1_i$.
In the construction for $x$, we instead use
\[
\Union_{j\in\phi[k_1]}p'_j\resto{\close{U}'_j-U'_{j'}}\colon\,\close{U}_{\phi(k_0)}\longto\Union_{j\in\phi[k_1]}J^0_j,
\]
where $j'$ denotes the element of $\phi[k_1]\sub[k_0]$, previous to
$j$.
Moreover, whenever we push-forward an algebra to $J^1_i$ in the
original construction, we instead push the corresponding algebra
forward to $[s'_{\phi i},x_i]$.
The rest of the construction will be unchanged.
The construction is functorial on $\Map_{/[k_0]}(\phi,J^0)$.

Note that, for the point $x=\phi^*(v^0)\in\Map_{/[k_0]}(\phi,J^0)$, we
obtain the $k_1$-simplex $\phi^*\int^c_{U,p,J^0}A$.
In particular, if we are given a $1$-simplex of $X_\dot$ in the
component for $(\phi,U)\in N_1\cat{X}$, specified by $p$,
$J=(J^\iota)_{\iota\in[1]}$ as above, and a map $\alpha\colon
v^1\equivto\phi^*(v^0)$ in $\Map_{/[k_0]}(\phi, J^0)$, then $\alpha$
induces an equivalence
$\int^c_{\phi^*U,\phi^*p,J^1}A\equivto\phi^*\int^c_{U,p,J^0}A$, or
equivalently, a (Cartesian) map
$\int^c_{\phi^*U,\phi^*p,J^1}A\to\int^c_{U,p,J^0}A$ in
$\laxcol_{\Simp^\op}N_\dot\cat{A}$, covering the map $\phi$ in
$\Simp$.

Using this, for every fixed $\kappa\ge 1$, we construct the map
$X_\kappa\to N_\kappa\laxcol_{\Simp^\op}N_\dot\cat{A}$ over
$N_\kappa\Simp$ as follows.
Namely, let a point of $X_\kappa$ in the component for $(\phi,U)\in
N_\kappa\cat{X}$ be specified by $(p,J,\alpha)$,
$J=(J^\iota)_{\iota\in[\kappa]}$,
$\alpha=(\alpha^\iota)_{\iota\in[\kappa]-\{0\}}$, $\alpha^\iota\colon
v^\iota\equivto\phi_\iota^*(v^{\iota-1})$, as before.
Then applying the above construction to the $1$-faces of this
$\kappa$-simplex connecting the adjacent vertices, we obtain a
sequence of maps
\begin{equation}\label{eq:nerve-of-compactly-supported-homology}
\int^c_{\psi_\kappa^*U,\psi_\kappa^*p,J^\kappa}A\longto\cdots\longto\int^c_{\psi_\iota^*U,\psi_\iota^*p,J^\iota}A\longto\cdots\longto\int^c_{U,p,J^0}A
\end{equation} 
in $\laxcol_{\Simp^\op}N_\dot\cat{A}$.
Then using composition of maps in $\laxcol_{\Simp^\op}N_\dot\cat{A}$,
we obtain a $\kappa$-simplex
$[\kappa]\to\laxcol_{\Simp^\op}N_\dot\cat{A}$.
Since the construction is functorial on $X_\kappa$, we obtain a
desired map.

Moreover, the collection over $[\kappa]\in\Simp$ of these maps
$X_\kappa\to N_\kappa\laxcol_{\Simp^\op}N_\dot\cat{A}$ over
$N_\kappa\Simp$, has a functoriality in $[\kappa]\in\Simp$,
coming from our construction of the simplicial structure of $X_\dot$
by compositions (and their associativity) of maps in the fundamental
groupoids of spaces of increasing sequences we used, and the
construction of the sequence
\eqref{eq:nerve-of-compactly-supported-homology} by functoriality on
these groupoids.
This completes the construction of a functoriality of
$U\mapsto\int^c_UA$ on $\Open(M)^\op$.

\subsection{Symmetric monoidality}
\label{sec:monoidality}
Next, we would like to give the compactly supported homology functor
$\int^c_?A\colon\Open(M)^\op\to\cat{A}$
a natural symmetric monoidal structure.
Recall from Section \ref{sec:symmetric-monoidal}
that, for us, this means extending the functor to a map of the
functors $\Fin_*\to\Cat$ (i.e., ``pre-$\Gamma$-categories'').
The pre-$\Gamma$-category to be the target here is the underlying
functor $S_+\mapsto\cat{A}^S$ of the symmetric monoidal category
$\cat{A}$.
The source is the functor on $\Fin_*$ defined by
$S_+\mapsto\Open^{(S)}(M)^\op$, where $\Open^{(S)}(M)$ is the full
subposet of $\Open(M)^S$ consisting of families $U=(U_s)_{s\in S}$ of
pairwise disjoint open submanifolds of $M$, indexed by $S$.

Thus, it suffices to show that our construction of the functor
$\Open(M)^\op\to\cat{A}$ extends to $\Open^{(S)}(M)^\op\to\cat{A}^S$
in a way functorial in $S_+$.
This can be done by defining
$\cat{X}^{(S)}:=\laxcol_{\Simp^\op}N_\dot\Open^{(S)}(M)^\op$, and
concretely constructing simplicial spaces $X^{(S)}$ and maps
\[
N_\dot\cat{X}^{(S)}\longequivfrom
X^{(S)}_\dot\longto
N_\dot\laxcol_{\Simp^\op}(N_\dot\cat{A})^S=N_\dot\laxcol_{\Simp^\op}N_\dot\cat{A}^S
\]
over $N_\dot\Simp$, functorially in $S_+$, extending the previous
construction from the case where $S$ is one point.

We define the simplicial space $X^{(S)}_\dot$ so the $\kappa$-th space
is the coproduct over $N_\kappa\cat{X}^{(S)}$ of the following
(contractible) spaces.
Namely, let $(\phi,U)\in N_\kappa\cat{X}^{(S)}$, where $\phi$ is a
$\kappa$-nerve \eqref{eq:nerve-of-simplices} in $\Simp$, and $U$ is a
$k_0$-simplex $U_0\into\cdots\into U_{k_0}$ of $N_\dot\Open^{(S)}(M)$,
where $U_i=(U_{is})_{s\in S}$ is a family of disjoint open submanifolds of
$M$, and each $U_{*s}$ is a $k_0$-simplex of $N_\dot\Open(M)$.
Then we let the component of $X^{(S)}_\kappa$ corresponding to
$(\phi,U)$ be the same as the component of $X_\kappa$ corresponding to
$(\phi,\Disj_SU)$ of $N_\kappa\cat{X}$.
As we have observed before, these spaces are contractible, so the
projection $X^{(S)}_\kappa\to N_\kappa\cat{X}^{(S)}$ is an
equivalence.
Note that a concrete description of the unique lift of the simplicial
structure of $N_\dot\cat{X}^{(S)}$ to $X^{(S)}_\dot$, is obtained by
pulling back the concrete description of the lift of the simplicial
structure along $X_\dot\to N_\dot\cat{X}$.

Moreover, since the map $X^{(S)}_\dot\to N_\dot\cat{X}^{(S)}$ is an
equivalence, the functoriality in $S_+$ of $N_\dot\cat{X}^{(S)}$ lifts
for $X^{(S)}_\dot$ uniquely.
We will want to use the following concrete description of a lift.
Namely, suppose given a map $f\colon S_+\to T_+$ in $\Fin_*$.
Then for a simplex $U$ of $N_\dot\Open^{(S)}(M)$, we have
$\Disj_Tf_!U=\Disj_{f^{-1}T}U$ for the $\Gamma$-structure map $f_!$ of
$\Open(M)$.
Therefore, if $(\phi,U)$ is a simplex of $N_\kappa\cat{X}^{(S)}$, then
the data for specifying a point in the corresponding component of
$X^{(S)}_\kappa$, namely, in the component of $X_\kappa$ corresponding
to $(\phi,\Disj_SU)\in N_\kappa\cat{X}$, can be restricted to
$\Disj_Tf_!U$, to specify a point in the component of $X^{(T)}_\kappa$
corresponding to $f_!(\phi,U)=(\phi,f_!U)\in N_\kappa\cat{X}^{(T)}$.
This defines a map $X^{(S)}_\kappa\to X^{(T)}_\kappa$, which is
functorial in $\kappa$ by our concrete description of the simplicial
structures.
By taking the resulting map $N_\kappa\cat{X}^{(S)}\to
N_\kappa\cat{X}^{(T)}$ to be the structure map $f_!$, we obtain a
concrete description of a $\Gamma$-structure of $X_\dot$, lifting
that of $N_\dot\cat{X}$.

In order to construct a map
$X^{(S)}_\dot\to N_\dot\laxcol_{\Simp^\op}(N_\dot\cat{A})^S$ over
$N_\dot\Simp$, note that
$N_\kappa\laxcol_{\Simp^\op}(N_\dot\cat{A})^S$ is the collection of
the components of
$N_\kappa\left(\laxcol_{\Simp^\op}N_\dot\cat{A}\right)^S$ lying over
the diagonal of $N_\kappa\Simp^S$.
Thus, we shall construct a map
$X^{(S)}_\dot\to
\left(N_\dot\laxcol_{\Simp^\op}N_\dot\cat{A}\right)^S$ landing in
these components.
It suffices to describe the component for each $s\in S$ of this map.
We let it be the composite
\[
X^{(S)}_\dot\xlongrightarrow{\pr_s}X_\dot\xlongrightarrow{\int^c_?A}N_\dot\laxcol_{\Simp^\op}N_\dot\cat{A},
\]
where $\pr_s$ is the map corresponding to the inclusion $\{s\}_+\into
S_+$ in the $\Gamma$-structure of $X_\dot$, which has been concretely
described above.
This gives a map $X^{(S)}_\dot\to
N_\dot\laxcol_{\Simp^\op}(N_\dot\cat{A})^S$ over
$N_\dot\Simp$, as desired.

Finally, we want functoriality of these maps in $S_+$.
However, using the concrete description of the $\Gamma$-structure of
$X_\dot$, this results immediately from the
symmetric monoidality of $A$ and the augmentation map of $A$, and our
assumption on the monoidal structure of $\cat{A}$ (see Section
\ref{sec:manifold-algebra}).

\section{Koszul duality for factorisation algebras}
\label{sec:duality}
\setcounter{subsection}{-1}
\setcounter{equation}{-1}
\subsection{The Koszul dual of a factorisation algebra}
\label{sec:construction}

We have seen in the previous section that compactly supported homology
with coefficients in an augmented locally constant factorisation
algebra $A$, is contravariantly functorial, and symmetric monoidal, in
the open submanifolds.
Let us denote this symmetric monoidal functor by $A^+$.
Namely, for an open submanifold $U$ of the manifold $M$ on which $A$
is defined, we denote $A^+(U):=\int^c_UA$.

We can restrict this coalgebra on $\Open(M)$ to $\Dis(M)$, and
consider it as a coalgebra on the multicategory $\Disk(M)$.
Let us denote this coalgebra by $A^!$.

\begin{definition}\label{def:koszul-dual}
Let $\cat{A}$ be a symmetric monoidal category which is closed under
sifted colimits, and whose monoidal multiplication preserves sifted colimits.
Let $A$ be an augmented locally constant factorisation algebra on a
manifold $M$, taking values in $\cat{A}$.
Then the \kore{Koszul dual} of $A$ is defined as the augmented locally
constant coalgebra $A^!$ on $\Disk(M)$, taking values in $\cat{A}$.
\end{definition}

Even though $A^!$ came from a functor $A^+$, $A^+$ may not in general
satisfy any reasonable descent property.
However, if the symmetric monoidal structure of
$\cat{A}$ behaves well with both sifted colimits \emph{and} sifted
limits, then the results of \cite[Section 2]{descent}, in particular,
Theorem 2.11, can be applied in $\cat{A}^\op$.
In particular, there is a universal way to extend $A^!$ to a functor
on $\Open(M)$ satisfying factorising descent.
In this case, one may expect $A^+$ to be close to the functor extended
from $A^!$ by descent.

One of the cases is where $\cat{A}$ is the category $\Space^\op$ of
the opposite spaces with the coCartesian symmetric monoidal structure.
In this case, $A^+$ satisfies factorising descent as often as one may
expect.
This is indeed Lurie's ``nonabelian Poincar\'e duality'' theorem
\cite{higher-alg} (which is closely related to earlier results of
Segal \cite{segal-config}, McDuff \cite{mcduff} and Salvatore
\cite{salvato}) as we shall discuss in Section
\ref{sec:poincare-duality}.

Thus, compactly supported factorisation homology gives a context
generalising the context for this theorem.
As another case where the monoidal structure behaves well with sifted
colimits and sifted limits, we shall analyse in
Section \ref{sec:poincare-duality}, the case where $\cat{A}$ is
$\Space$ with the Cartesian symmetric monoidal structure.
We will find in this case that the theorem of Lurie's type (which in
fact is equivalent to Lurie's theorem) admits a
refinement which does not seem to exist in the opposite context (see
Remark \ref{rem:opposite-gromov}).

In the later sections, we shall describe a result in which the
symmetric monoidal structure is \emph{not} required to behave well
with sifted limits, but the behaviour of the functor $A^+$ can still
be nice thanks to some additional structure on the target category.

\bigskip

For the remainder of this section, we shall see some simple examples
of the Koszul dual coalgebras.
More specifically, we shall see instances of the following,
easy consequence of the constructions.
Let us denote the category of augmented locally constant
factorisation algebras by $\Alg_{M,*}(\cat{A})$, and the category of
augmented locally constant coalgebras on $\Disk(M)$ by
$\Coalg_{M,*}(\cat{A})$.
\begin{proposition}\label{prop:relating-koszul-duals}
Let $\cat{A}$ and $\cat{B}$ be symmetric monoidal category which are
closed under sifted colimits, and whose
symmetric monoidal multiplication preserves sifted colimits.
Let $F\colon\cat{A}\to\cat{B}$ be a symmetric monoidal functor which
preserves sifted colimits.
Then, the canonical map filling the square
\[\begin{tikzcd}
\Alg_{M,*}(\cat{A})\arrow{d}[swap]{F}\arrow{r}{\blank^!}&\Coalg_{M,*}(\cat{A})\arrow{d}{F}\\
\Alg_{M,*}(\cat{B})\arrow{r}{\blank^!}&\Coalg_{M,*}(\cat{B})
\end{tikzcd}\]
is an equivalence.
\end{proposition}

One example of a functor $F$ as in the proposition is given by the
Koszul duality functor.
Namely, let $\cat{A}$ be a symmetric monoidal category which is closed
under sifted colimits, and whose symmetric monoidal multiplication preserves
sifted colimits.
Then we can apply the proposition to the functor
$\blank^!\colon\Alg_{N,*}(\cat{A})\to\Coalg_{N,*}(\cat{A})$, where $N$
is any manifold without boundary.

Let us apply the proposition on another manifold $M$.
Then we obtain that the canonical map filling the square
\begin{equation}\label{eq:iterating-koszul}
\begin{tikzcd}
\Alg_{M,*}(\Alg_{N,*})\arrow{d}[swap]{\blank^!}\arrow{r}{\blank^!}&\Coalg_{M,*}(\Alg_{N,*})\arrow{d}{\blank^!}\\
\Alg_{M,*}(\Coalg_{N,*})\arrow{r}{\blank^!}&\Coalg_{M,*}(\Coalg_{N,*}),
\end{tikzcd}
\end{equation}
where we have dropped $\cat{A}$ from our notation, is an equivalence.

In fact, it is immediate to see that the diagonal map here is
canonically equivalent to the composite
\[
\Alg_{M\times N,*}\xrightarrow{\blank^!}\Coalg_{M\times
  N,*}\xrightarrow{\mathrm{restriction}}\Coalg_{M,*}(\Coalg_{N,*}),
\]
through the identification $\Alg_{M\times
  N,*}\equivto\Alg_{M,*}(\Alg_{N,*})$ by the restriction functor
following from \cite[Theorem 3.14]{descent}.

For example, the Koszul dual on $\R^n$ can be understood as the result
of iteration of taking the Koszul dual on $\R^1$.
The definition in terms of the compactly supported homology gives a
coordinate-free description of the same thing.

\bigskip

We shall give a few more examples of symmetric monoidal functors to
which Proposition \ref{prop:relating-koszul-duals} applies.
(We will be very far from being comprehensive (and from being most
general).
Our purpose is to discuss just a few of many examples for
illustration.)
\begin{example}
Fix a base field (in the usual discrete sense) of characteristic $0$,
and let $(\Lie,\directsum)$ and $(\Mod,\directsum)$ denote the
symmetric monoidal category of Lie algebras and of modules
respectively, over the base field, with the symmetric monoidal
structure given by the direct sum operations.
(We may consider dg Lie algebras and dg modules (chain complexes).)
When we do not specify the symmetric monoidal structure in the
notation as above, let us understand we are taking the symmetric
monoidal structures given by the tensor product (over the base field).

Then any of the symmetric monoidal functors appearing in the following
commutative diagram preserves sifted colimits.
\[\begin{tikzcd}
{}&(\Mod,\directsum)\arrow{ld}[swap]{=}\arrow[hook]{d}\arrow{r}{\susp}[swap]{\equiv}
&(\Mod,\directsum)\arrow{d}[swap]{\Sym}\\
(\Mod,\directsum)
&(\Lie,\directsum)\arrow{l}{\mathrm{forget}}\arrow{r}[swap]{C_\dot}\arrow{d}{U}
&\Coalg_{\Com,*}\arrow{r}{\mathrm{forget}}\arrow{d}{\mathrm{forget}}
&\Coalg_\Com\arrow{d}{\mathrm{forget}}\\
&\Alg_{E_n,*}\arrow{r}[swap]{\blank^!}
&\Coalg_{E_n,*}\arrow{r}[swap]{\mathrm{forget}}
&\Coalg_{E_n}\arrow{r}[swap]{\mathrm{forget}}
&\Mod,
\end{tikzcd}\]
where $\susp=\blank[1]$ is the suspension functor,
$(\Mod,\directsum)\into(\Lie,\directsum)$ is the inclusion of Abelian
Lie algebras, $\Com$ stands for ``commutative'' ($=E_\infty$),
$C_\dot$ is the Lie algebra homology functor, and $U$
is the enveloping $E_n$-algebra functor (where $n\ne\infty$).
\end{example}

Proposition \ref{prop:relating-koszul-duals} gives a result
on comparison of the Koszul dual coalgebras through any of these
functors.

\begin{corollary}
The functor $\blank^!\colon\Alg_M(\Lie,\directsum)\to\Sh_M(\Lie)$ is
an equivalence.
Compactly supported homology and cohomology satisfy descent.
\end{corollary}
\begin{proof}
$\blank^!\colon\Alg_M(\Mod,\directsum)\to\Coalg_M(\Mod,\directsum)$ is
the Verdier functor $\Cosh_M(\Mod)\to\Sh_M(\Mod)$, and is easily seen
to be an equivalence by looking at what it does at the level of
stalks.

The result follows since the functor $\Lie\to\Mod$ reflects
equivalences.
\end{proof}

The functor
\[
\Sh_M(\Lie)\xrightarrow[\equiv]{\blank^!}\Alg_M(\Lie,\directsum)\xrightarrow{C_\dot}\Alg_M(\Mod)
\]
is particularly interesting since by the work of Costello and Gwilliam
\cite{cg}, \cite{gwilliam}, for a particular sheaf $\lie{g}$ of Lie
algebras over $\C[[\hbar]]$ (the Heisenberg Lie algebras), the
factorisation algebra $C_\dot(\lie{g}^!)$ is the factorisation algebra
of observables of the (deformation) quantisation of a free classical
field theory, in the framework of \cite{cg}.

Proposition \ref{prop:relating-koszul-duals} applies to this functor.
Note that if $M$ is a Euclidean space, then the category of sheaves of
Lie algebras is just the category of Lie algebras, since our sheaves
are assumed to be locally constant.

The compactly supported homology of the factorisation algebra
$C_\dot(\lie{g}^!)$ is easy to describe.
Namely, we have
\[
\int^c_UC_\dot(\lie{g}^!)\:=\:C_\dot\int^c_U\lie{g}^!\:=\:C_\dot(\lie{g}(U)).
\]
More on this will be discussed in Section \ref{sec:variant}.

\subsection{Descent properties of compactly supported factorisation
  homology}
\label{sec:poincare-duality}
In this section, we examine the Koszul duality in particular
situations where the monoidal structure preserves both sifted colimits
and sifted limits variable-wise.

As a first example, we observe that Lurie's ``nonabelian Poincar\'e
duality'' theorem \cite{higher-alg} we have introduced in Section
\ref{sec:koszul-intro} is a theorem about the Koszul duality for
factorisation algebras.
Indeed, there, we have stated the theorem in terms of a functor $E^+$
obtained from a sheaf $E$ of spaces, by taking compactly supported
cohomology.
The sheaf $E$ may be a locally constant sheaf of spaces in the
infinity $1$-categorical sense, in order for the theorem to make
sense, and to be true.
Such $E$ can be identified exactly with a locally constant
factorisation algebra $A$ taking values in $\Space^\op$.
Then the prealgebra $E^+$ in $\Space$ gets identified with the
precoalgebra $A^+$ in $\Space^\op$ we have defined in Section
\ref{sec:construction}.

Thus, Lurie's theorem is along the line of discussions we have made
after Definition \ref{def:koszul-dual} in Section
\ref{sec:construction}.

See \cite{higher-alg} for the relation of this to the classical
`Abelian' or `stable' Poincar\'e duality theorem.
In the classical context, the role of the Koszul duality is played by
the Verdier duality.

Another interesting point mentioned in \cite{higher-alg} is that at a
point of $M$, the stalk of $A^!$ is the $n$-fold based loop space of
the stalk of $E$, and the structure of a factorisation algebra of
$A^!$ is extending the structure of an $E_n$-algebra of the $n$-fold
loop space.
Namely, the Koszul duality construction in the current context is
globalising the looping functor in the context of the classical theory
of iterated loop spaces.
We shall next consider a globalisation of the \emph{delooping}
functor.

\bigskip

We go to the opposite context, and consider the case where the algebra
$A$ takes values in $\Space$, the category of spaces with the
Cartesian symmetric monoidal structure.

The following is a version of non-Abelian duality theorem in this
context.
It can be proved more or less similarly to Lurie's theorem.
However, this proposition can be also deduced from
Lurie's theorem, and vice versa.
\begin{proposition}
If every stalk of $A$ is group-like as an $E_1$-algebra, then $A^+$ is
a locally constant factorisation algebra in $\Space^\op$.
In particular, the
map $\int^c_MA\to\int_MA^!$ is an equivalence.
\end{proposition}

\smallskip

In the present context, the formal part of the proof of 
Gromov's h-principle applies, and we obtain the following.
In the opposite context, there does not seem to be a similar theorem
(at least in an interesting way).
See Remark \ref{rem:opposite-gromov}.

\begin{theorem}\label{thm:gromov}
Let $A$ be a locally constant factorisation algebra of spaces on a
manifold $M$.
Then the canonical map
$\int^c_MA\to\int_MA^!$($\equivwith\Gamma(M,A^!)$, derived sections)
of spaces is an equivalence if no connected component of $M$ is a
closed manifold (i.e., if $M$ is ``open'').
\end{theorem}
\begin{proof}
Note that the association $U\mapsto\Gamma(U,A^!)$ is the universal
locally constant sheaf associated to the locally constant presheaf
$A^+$, and the map $\int^c_MA\to\int_MA^!$ is the map on the global
sections of the universal map.

Take a handle body decomposition of $\close{M}$ involving no handle
of index $n:=\dim M$, and for this decomposition,
$A^!:U\mapsto\Gamma(U,A^!)$ satisfies excision for every handle
attachment.

We first prove that excision for attachment of a handle of index
smaller than $n$ is satisfied by $A^+$ as well.
More generally, suppose $W$ is an open submanifold of $M$ which as a
manifold by itself, is given as the interior of a
compact manifold $\close{W}$, and let $\{U, V\}$ be a cover of
$\close{W}$ by two open submanifolds (possibly with boundary) which
has a diffeomorphism $U\intersect V\equivto N\times\R^1$ for a
($n-1$)-dimensional manifold $N$, compact with boundary.
In the case of handle body attachment (so $V$, say is the attached
handle, and $\close{W}$ is the result of attachment), $N$ is of the
form $S^{i-1}\times\close\R^{n-i}$, where $i$ is the index of the
handle.

Let us denote $\interior{U}=U\intersect W$,
$\boundary
U=U\intersect\boundaryc{W}(\subsetneq\boundaryc{U})$,
and similarly for $V$ and $N$.
Then
$A^+(\interior{U}\intersect\interior{V})=A^+(\interior{N}\times\R^1)$
is an $E_1$-coalgebra, and let us assume that our choice of
orientation of $\R^1$ makes $A^+(\interior{U})$ a right, and
$A^+(\interior{V})$ a left module respectively, over this
$E_1$-coalgebra.
Then we want to show that the restriction map
\[
A^+(W)\longto
A^+(\interior{U})\cotensor_{A^+(\interior{U}\intersect\interior{V})}A^+(\interior{V}),
\]
where the target denotes the cotensor product, is an equivalence.

Recall that
\[
A^+(W)=A(W)\tensor_{A(\boundaryc{W})}\unity.
\]
Our idea is to apply the excision property of $A$ to $A(W)$
and $A(\boundaryc{W})$ in a compatible way.
That is, using the decomposition of $\close{W}$ into $U$ and $V$, and
its restriction to the boundary (or rather, a collar of
$\boundaryc{W}$), we obtain identifications
\[
A(W)=A(\interior{U})\tensor_{A(\interior{U}\intersect\interior{V})}A(\interior{V})
\]
and
\[
A(\boundaryc{W})=A(\boundary
U)\tensor_{A(\boundary U\intersect\boundary V)}A(\boundary V)
\]
which are compatible with the actions at boundary.
Therefore, by denoting by $G$ the $E_1$-algebra
$A(\interior{U}\intersect\interior{V})\tensor_{A(\boundary
  U\intersect\boundary V)}\unity$, we obtain that
\[
A^+(W)=K\tensor_GL,
\]
where $K$ is the right $G$-module
$A(\interior{U})\tensor_{A(\boundary U)}\unity$,
and $L$ the similar left $G$-module corresponding to $V$.
This follows since the difference between the objects in question is
difference in the order in which to realise a bisimplicial object.

However, in terms of these algebra and modules, we see that
$A^+(\interior{U})\tensor
A^+(\interior{V})=K\tensor_G\unity\tensor_GL$, and
$A^+(\interior{U}\intersect\interior{V})=\unity\tensor_G\unity$, and
by inspecting the actions, we find  that the assertion of excision
is that the canonical map
\[
K\tensor_GG\tensor_GL\longto(K\tensor_G\unity)\cotensor_{\unity\tensor_G\unity}(\unity\tensor_GL),
\]
defined by algebra is an equivalence.
We state this as a lemma below,
and the proof of the lemma will complete the proof of the excision
property of $A^+$ for handle body attachment.
The proof of lemma will use the assumption $i\le n-1$.

\medskip

The proof can be now completed similarly to the proof of Theorem
\ref{thm:excision}, by induction on the number of handles in the
decomposition of $\close{M}$ we have been considering.
Namely, since both $A^+$ and $A^!$ satisfy excision for handle body
attachment of index smaller than $\dim M$, by induction, it suffices
to prove that the value of these functors agree on open submanifolds
of $M$ diffeomorphic to either a disk or $\boundary D^i\times
D^{n-i+1}$, where $i\le n-1$.

The case of a disk is by the definition of $A^!$.
The latter case follows from the excision for the handle body
attachment of index smaller than $n$, since $\boundary D^i\times
\close{D}^{n-i+1}$ can be obtained by attaching a handle of index
$i-1$ to a disk.
\end{proof}

Let us state and prove the lemma which was promised in one of the
steps in the proof.
To recall the notation,
$G$ is an $E_1$-algebra of spaces, and $K$ is a left,
and $L$ is a right, $G$-module respectively, both of spaces.
In the previous proof, we were in the situation where
$G=A(\interior{N}\times\R^1)\tensor_{A(\boundary N\times\R^1)}\unity$,
where $N=S^i\times\close\R^{n-i}$, and $n-i\ge 1$, and $A$ was a
locally constant factorisation algebra on $\interior{N}\times\R^1$.
In particular, the underlying space of $G$ was connected.

For the following lemma, we only need to assume that $G$ is
group-like, namely, the monoid $\pi_0(G)$ is in fact a group.
Since this assumption is satisfied for a connected $G$, the proof of
the following lemma closes the unfinished step of the previous proof.

\begin{lemma}\label{lemma:algebraic-unstable}
The canonical map
\[
K\tensor_GL\longto(K\tensor_G\unity)\times_{G^!}(\unity\tensor_GL)
\]
is an equivalence for every $K$ and $L$ if it is so for $K=G$
and $L=G$, namely if $G$ is group-like.
\end{lemma}
\begin{proof}
Consider the maps as a map over $G^!$.
Then the induced map on the fibres over the unique (up to homotopy)
point of $G^!=BG$ can be identified with the identity of $K\times L$.
\end{proof}

\begin{remark}\label{rem:opposite-gromov}
In the previous context where the target category $\cat{A}$
is $\Space^\op$ with the coCartesian symmetric monoidal structure,
there does not seem to be a result corresponding to
Theorem \ref{thm:gromov}, in an interesting way.

In fact, if $A$ is an arbitrary algebra in $\cat{A}$, namely, a
locally constant sheaf of spaces, then the map $\lift{A}\to A$ from
the stalk-wise $n$-connective cover (where $n$ is the dimension of our
manifold) induces an equivalence on the Koszul dual.
Therefore, for any open $U$, we have
$\int_UA^!=\Gamma_c(U,\lift{A})$.

However, it happens only rarely that the map $\lift{A}\to A$ induces
an equivalence on the space of compactly supported sections.
\end{remark}

\subsection{Poincar\'e duality for complete factorisation algebras}
\label{sec:poincare-duality-complete}
\setcounter{subsubsection}{-1}

In this section, let $\cat{A}$ be a symmetric monoidal complete
soundly filtered stable category with uniformly bounded sequential
limits, as defined in \cite[Sections 2, 3, Definition 4.3]{local}.

Let us first recall the following.

\begin{definition}[{\cite[Definition 4.16]{local}}]
\label{def:e-n-copositive}
An augmented $E_n$-algebra $A$ in $\cat{A}$ is said to be
\kore{positive} if its augmentation ideal belongs to $\cat{A}_{\ge
  1}$.

An augmented $E_n$-coalgebra $C$ in $\cat{A}$ is said to be
\kore{copositive} if there is a uniform bound $\omega$ for loops
\cite[Definition 2.42]{local} in $\cat{A}$ such that the augmentation
ideal $J$ of $C$ belongs to $\cat{A}_{\ge 1-n\omega}$.
\end{definition}

Let $M$ be a manifold (without boundary).
We shall prove a version of the non-Abelian Poincar\'e duality theorem
in $\cat{A}$, for factorisation algebras which is \emph{positive} in
the following sense.

\begin{definition}
An augmented factorisation algebra $A$ on
$M$ is said to be \kore{positive} if the stalk $A_x$ at every
point $x\in M$ is positive as an $E_n$-algebra, where $n=\dim M$.
\end{definition}

\begin{theorem}\label{thm:poincare-duality-complete}
Let $A$ be a positive augmented locally constant factorisation
algebra on $M$, valued in $\cat{A}$ as above.
Then, $A^+$, defined by compactly supported factorisation homology
(see Section \ref{sec:construction}), satisfies excision.
\end{theorem}

For the proof we use the following facts established in \cite{local}.

For an associative coalgebra $C$, let us denote by $-\cotensor_C-$ the
(co-)tensor product over $C$.

\begin{proposition}[{\cite[Proposition 4.7]{local}}]
\label{prop:commuting-tensor-and-cotensor}
Let $A$ be a positive augmented associative algebra, and $C$ a
copositive augmented associative coalgebra, both in $\cat{A}$.
Assume $A$ is positive, and $C$ is copositive (Definition
\ref{def:e-n-copositive}).
Let $K$ be a right $A$-module, $L$ an $A$--$C$-bimodule, and let $X$
be a left $C$-module, all bounded below.

Then the canonical map
\[
K\tensor_A(L\cotensor_CX)\longto(K\tensor_AL)\cotensor_CX
\]
is an equivalence.
\end{proposition}

\begin{theorem}[{\cite[Theorem 4.10]{local}}]
\label{thm:josh-complete-implies-koszul-complete}
Let $A$ be a positive augmented associative algebra in $\cat{A}$, and
$K$ be a right $A$-module which is bounded below.
Then the canonical map
$K\to(K\tensor_A\unity)\cotensor_{\unity\tensor_A\unity}\unity$ is an
equivalence (of $A$-modules).
\end{theorem}

\begin{proof}[Proof of Theorem \ref{thm:poincare-duality-complete}]
As in the proof of Theorem \ref{thm:gromov}, suppose that a situation
for excision is given as follows.

Namely, suppose $W$ is an open submanifold of $M$ which as a
manifold by itself, is given as the interior of a
compact manifold $\close{W}$, and let $\{U, V\}$ be a cover of
$\close{W}$ by two open submanifolds (possibly with boundary) which
has a diffeomorphism $U\intersect V\equivto N\times\R^1$ for an
($n-1$)-dimensional manifold $N$, compact with boundary.

Let us denote $\interior{U}=U\intersect W$,
$\boundary
U=U\intersect\boundaryc{W}(\subsetneq\boundaryc{U})$,
and similarly for $V$ and $N$.
Then
$A^+(\interior{U}\intersect\interior{V})=A^+(\interior{N}\times\R^1)$
is an $E_1$-coalgebra, and let us assume that our choice of
orientation of $\R^1$ makes $A^+(\interior{U})$ right, and
$A^+(\interior{V})$ left modules over this $E_1$-coalgebra.

Then, as before, we have an $E_1$-algebra
$B:=A(\interior{U}\intersect\interior{V})\tensor_{A(\boundary
  U\intersect\boundary V)}\unity$ (denoted by $G$ before),
a right $B$-module
$K:=A(\interior{U})\tensor_{A(\boundary U)}\unity$,
and $L$ the similar left $B$-module corresponding to $V$,
and the excision in question can be stated as that the canonical map
\[
K\tensor_BB\tensor_BL\longto(K\tensor_B\unity)\cotensor_{\unity\tensor_B\unity}(\unity\tensor_BL)
\]
defined by algebra, where the target denotes the cotensor product, is
an equivalence.

In order to prove this, it suffices to note that $B$ is a positive
augmented algebra.
Indeed, it follows that the map above is an equivalence by the above
results of \cite{local}.
\end{proof}

\begin{remark}\label{rem:excision-was-enough}
Note that even if we do not assume that the monoidal operations
preserve sifted colimits, the proof works if the prealgebra satisfies
excision.
\end{remark}

\smallskip

Theorem has the following interesting consequence.
A version of this was obtained earlier by Francis
(\cite{glanon,francis}).
To give a context, let us recall that if $A$ is an $E_n$-algebra,
then $A$ is an $n$-dualisable object of the Morita ($n+1$)-category
$\Alg_n(\cat{A})$ (see Lurie \cite{tft}).
A concrete description of the associated topological field theory can
be outlined as follows \cite{tft}.
$A$ defines a locally constant factorisation algebra on
every framed $n$-dimensional manifold which can be considered to be
`(globally) constant at $A$' in a sense.
In particular, for $k\le n$, if $M$ is a ($n-k$)-dimensional manifold
equipped with a framing of $M\times\R^k$ (so $M$ might be a
($n-k$)-morphism in the $n$-dimensional framed cobordism category), then
we obtain an $E_k$-algebra by pushing forward the constant
factorisation algebra on $M\times\R^k$, along the projection
$M\times\R^k\to\R^k$.
Let us denote this $E_k$-algebra by $\int_MA$.
If $M$ is indeed a ($n-k$)-morphism in the $n$-dimensional framed
cobordism category, then $\int_MA$ interact with the factorisation
homology of $A$ over the manifolds appearing as its sources and
targets of all codimensions, in a certain specific way to make it a
($n-k$)-morphism in $\Alg_n(\cat{A})$.
Excision then implies that the association $M\mapsto\int_MA$ is
functorial with respect to the compositions in the cobordism category,
so this gives an $n$-dimensional fully extended framed topological
field theory.
The value of this theory for an $n$-framed point $p$ is indeed the
$E_n$-algebra $A$.

For $\cat{A}$ a complete soundly filtered stable with uniformly
bounded sequential limits, we have shown in \cite[Section 4.3]{local}
that the Koszul duality has a Morita theoretic functoriality.
Namely, we have constructed a positive augmented coalgebraic version
$\Coalg^\positive_n(\cat{A})$ of the higher Morita category, and have
shown that the construction of the Koszul duals gives a symmetric
monoidal functor
$\Alg^\positive_n(\cat{A})\to\Coalg^\positive_n(\cat{A})$, which is an
equivalence \cite[Theorem 4.22]{local}.
(The source here is the positive part of the Morita ($n+1$)-category
of augmented algebras.)
If $A$ is augmented and positive, then the field theory in
$\Alg^\positive_n(\cat{A})$ associated to $A$ corresponds via this
functor, to a theory in $\Coalg^\positive_n(\cat{A})$ associated to
$A^!$.

On the other hand, a consequence of
Theorem \ref{thm:poincare-duality-complete} is that there is an
$n$-dimensional fully extended topological field theory in
$\Coalg^\positive_n(\cat{A})$ which associates to an ($n-k$)-morphism
$M$ in the framed cobordism category, the compactly supported homology
$\int^c_{M\times\R^k}A$ (with suitable algebraic structure) of the
factorisation algebra on $M\times\R^k$, ``constant at $A$''.
(To be accurate, $\int^c_{M\times\R^k}A$ is a slightly more general
than what we have considered, in that $M$ is compact with boundary and
other higher codimensional corners.
However, since we are dealing only with constant coefficients, there
is almost no difficulty added in establishing their basic behaviour.)
The value for the $n$-framed point $p$ of this theory is the
$E_n$-coalgebra $\int^c_{\R^n}A=A^!$.
(See \eqref{eq:iterating-koszul}.)

The cobordism hypothesis implies that there is a unique
equivalence between these two theories in
$\Coalg^\positive_n(\cat{A})$ which fixes the common value for the
$n$-framed point.
Such an equivalence can be concretely seen by writing
\[
\int^c_{M\times\R^k}A=\int^c_{\R^k}\int_MA=\left(\int_MA\right)^!,
\]
where $\int_MA$ for all $M$ are understood to be equipped with the
algebraic structures to make them morphisms in
$\Alg^\positive_n(\cat{A})$.
In fact, the excision situations we need to consider to check the
functoriality of $\int^c_{M\times\R^k}A$ in the cobordism category,
all reduce to the situations we considered in checking
the Morita functoriality of the Koszul duality construction
\cite[Theorem 4.22]{local},
and we have used the identical arguments in both situations.

\bigskip

As another consequence of the Poincar\'e duality
theorem \ref{thm:poincare-duality-complete}, we obtain an equivalence
of categories from the Koszul duality on a manifold, generalising
the equivalence on $\R^n$ (Theorem \ref{thm:koszul-duality-e-n}).
This shows that the Koszul duality for factorisation algebras is a
non-Abelian extension of the Verdier duality.

Let us define the suitable category of augmented coalgebras.
Let $\Open^\loc_1(M)$ denote the following multicategory.
We define it by first defining its category of colours, and then
giving it a structure of a multicategory.

First let $\manc$ denote the discrete category whose object is a
``manifold without boundary'' in our convention stated in Section
\ref{sec:manifold-algebra},
and where maps are ``open embeddings'' as defined at the same place.
$\manc$ is a symmetric monoidal category under disjoint union.

Let $\manc^\loc$ denote the following, non-discrete version of
$\manc$.
Namely, its objects are the same as $\manc$, but we take the space of
morphisms to be the \emph{space} of open embeddings.
Let $\manc_1$ denote the full subposet of $\manc$
consisting of manifolds with exactly one connected component.

We define the category of colours of $\Open^\loc_1(M)$ as
$\manc_1\times_{\manc^\loc}\manc^\loc_{/M}$.

This (as any category) is the category of colours of a multicategory
where a multimap is simply a family of maps.
The structure of a multicategory we consider on
$\manc_1\times_{\manc^\loc}\manc^\loc_{/M}$ to define
$\Open^\loc_1(M)$, is a restriction of this structure of a
multicategory, where we require the family of maps in $\manc_1$
specified as a part of the data of a family of maps in
$\manc_1\times_{\manc^\loc}\manc^\loc_{/M}$, to be pairwise
disjoint (over the interiors).
Namely, given a finite set $S$ and a family $U=(U_s)_{s\in S}$ of
objects, and an object $V$, a multimap $U\to V$ is by definition an
open embedding $\coprod_SU\into V$ together with for each $s\in S$, a
path in the space $\Emb(U_s,M)$, from the defining embedding
$U_s\into M$ to the composite $U_s\into V\into M$.

Let us denote by $\Coalg_M(\cat{A})_\positive$ the category of
\emph{copositive} augmented coalgebras on $\Open^\loc_1(M)$,
valued in $\cat{A}$ which satisfies excision, where by
\kore{copositive}, we mean that every stalk of the augmented
coalgebra is copositive as an augmented $E_n$-coalgebra (Definition
\ref{def:e-n-copositive}).

\begin{theorem}\label{thm:verdier}
Let $\cat{A}$ be a symmetric monoidal complete soundly filtered
stable category with uniformly bounded sequential limits.
Assume that the monoidal operations preserves sifted colimits
(variable-wise).
Then the functor
\[
\blank^+\colon\Alg_M(\cat{A})_\positive\longto\Coalg_M(\cat{A})_\positive
\]
is an equivalence.
\end{theorem}
\begin{proof}
The inverse is given by taking compactly supported `co'-homology.
Namely, if $C\in\Coalg_M(\cat{A})_+$ (namely, is a \emph{copositive}
augmented coalgebra), then the prealgebra $C^+$
defined by $C^+(U)=\int^c_UC$ (the definition of Section
\ref{sec:compactly-supported-homology} applied in $\cat{A}^\op$ using
the copositivity) satisfies excision by the proof similar to
the proof of Theorem \ref{thm:poincare-duality-complete}, and hence is
a (positive augmented) factorisation algebra.

One checks that this functor is indeed inverse to the given functor,
by looking at what these functors do to the stalks of algebras and
coalgebras.
At the level of stalks, the functor is the Koszul duality construction of
$E_n$-algebras and coalgebras, done by iteration of $E_1$-Koszul
duality constructions.
The result now follows from Theorem \ref{thm:koszul-duality-e-n},
which is obtained by iterating the $E_1$-case of it 
\cite[Sections 4.0, 4.1]{local}.
\end{proof}

In particular, the functor $M\mapsto\Coalg_M(\cat{A})_\positive$
satisfies descent for (effectively) factorising l-nice bases as in
\cite[Theorem 2.26]{descent}.
If we had known this descent property first, then Theorem could have
been deduced from the local case Theorem \ref{thm:koszul-duality-e-n}.
A direct proof of the descent property might involve some interesting
extension of our methods developed in \cite[Section 2]{descent}.

\begin{remark}
If the assumption of preservation of sifted colimits by the
monoidal operations is not satisfied, then a positive augmented
algebra on $\Open^\loc_1(M)$ satisfying excision still seems to be a
meaningful notion.
Theorem remains true for such objects.
See also Remark \ref{rem:excision-was-enough} and
\cite[Remark 4.27]{local}.
\end{remark}

\subsection{Example of a positive factorisation algebra}
\label{sec:example-positive-algebra}
\setcounter{subsubsection}{-1}
In this section, we show that a factorisation algebra in any
reasonable symmetric monoidal stable category canonically gives rise
to a positive factorisation algebra in the symmetric monoidal category
of filtered objects there \cite[Section 3.2]{local}.
It will follow that the Poincar\'e duality theorem applies to the
completion of this positive filtered factorisation algebra
(Corollary \ref{cor:poincare-for-completed-algebra} below).

Let $\cat{A}$ be a symmetric monoidal stable category.
Assume $\cat{A}$ has all sequential limits in it.
Let $A$ be an augmented factorisation algebra  in $\cat{A}$ on a
manifold $M$.
We would like to define a filtered factorisation algebra $F_\dot A$.
(Recall from Section \ref{sec:manifold-algebra} that we are assuming
that the monoidal multiplication preserves colimits variable-wise.)

Let us first introduce some notations.
Let $\Fin$ denote the category of finite sets.
Let $\Fin^\hitan$ denote the category of finite sets, with
surjections as maps.
For an integer $r\ge 0$, let $\Fin_{\ge r}$ denote the full
subcategory of $\Fin^\hitan$ consisting of objects $T$ with at least
$r$ elements.

Given any category $\cat{C}$ equipped with a functor
$\cat{C}\to\Fin$, define
\[
\cat{C}_{\ge r}\::=\:\cat{C}\times_\Fin\Fin_{\ge r}
\]
whenever the choice of the functor to $\Fin$ is understood.
This is a full subcategory of
\[
\cat{C}^\hitan\::=\:\cat{C}_{\ge 0}\:=\:\cat{C}\times_\Fin\Fin^\hitan.
\]

For example, for any manifold $M$, from the functor
$\pi_0\colon\Dis(M)\to\Fin$ which associates to a disjoint union of
disks, the finite set of its components, we obtain a poset
$\Dis^\hitan(M):=\Dis(M)^\hitan$ and its full subposets $\Dis_{\ge
  r}(M):=\Dis(M)_{\ge r}$.

\begin{example}\label{ex:filtration-of-disks}
$\Dis_{\ge 0}(U)=\Dis^\hitan(U)=\Dis_{\ge 1}(U)\disj\{\kara\}$.

If $U$ is empty, then $\Dis_{\ge\dot}(U)$ is the unit `filtered'
category, namely $\Dis_{\ge 0}(U)=*$, and $\Dis_{\ge r}(U)=\kara$ for
$r\ge 1$.
\end{example}

We shall often represent an object of $\Dis^\hitan(M)$ as a pair $(T,D)$,
where $T\in\Fin^\hitan$, and $D=(D_t)_{t\in T}$ is a family of disjoint
disks in $M$, indexed by $T$.

\bigskip

Let us start a construction of $F_\dot A$ for an augmented
factorisation algebra $A$.

Let us denote by $\reduce{A}$ the reduced version of $A$, which
can be considered as a symmetric monoidal functor on $\Dis^\hitan(M)$.
Namely, given a pair $(T,D)\in\Dis^\hitan(M)$, 
$\reduce{A}$ associates to it the object
\[
\reduce{A}(D)\::=\:\Tensor_{t\in T}I(D_t),
\]
where $I:=\Fibre[\varepsilon\colon A\to\unity]$ (the section-wise
fibre).

For an open submanifold $U$ of $M$, we define
\[
F_rA(U)=\colim_{\Dis_{\ge r}(U)}\reduce{A}.
\]
In other words, if $\Z$ denotes the integers made into a category by
their order as
\[
\cdots\longfrom r\longfrom r+1\longfrom\cdots,
\]
then the functor $\Z\ni r\mapsto F_rA(U)$ is the left Kan
extension of $\reduce{A}$ along the functor
$\Dis^\hitan(U)\xrightarrow{\pi_0}\Fin^\hitan\xrightarrow{\cardinality}\Z$,
where ``$\cardinality$'' takes the cardinality of finite sets.

\begin{example}\label{ex:positivity-of-canonical-filtration}
It follows from Example \ref{ex:filtration-of-disks} that the values
of the augmentation ideal $IF_\dot A=\Fibre[\varepsilon\colon F_\dot
A\to\unity]$ are positive (i.e., ``$\ge 1$'') in the filtration.

If $U$ is empty, then $F_\dot A(U)=\unity$.

If $U$ is a disk, then $\Dis_{\ge 1}(U)$ has a maximal element (a
terminal object), so $F_1A(U)=\reduce{A}(U)$ and $F_0A(U)=A(U)$.
\end{example}

We would like to prove that $F_\dot A$ defines a locally constant
factorisation algebra of filtered objects.

\begin{lemma}\label{lem:filtered-factorisation}
Let $U$, $V$ be manifolds.
Then the functor
\[
\Dis_{\ge r}(U\kocoprod V)\longfrom\laxcol_{i+j\ge r}\Dis_{\ge
  i}(U)\times\Dis_{\ge j}(V)
\]
given by taking disjoint unions, has a left adjoint.
In particular, it is cofinal.
\end{lemma}
\begin{proof}
The left adjoint is given as follows.
Let $(T,D)$ be an element of $\Dis_{\ge r}(U\kocoprod V)$.
Then $T$, $D$ can be written uniquely as
\[
T=T'\kocoprod T'',\quad D=D'\kocoprod D''
\]
where $D'$ (resp.~$D''$) is a collection of disks in $U$ (resp.~$V$),
indexed by $T'$ (resp.~$T''$).

From these, we obtain an element $(T',D')$ (resp.~$(T'',D'')$)
of $\Dis_{\ge\noudo T'}(U)$ (resp.~$\Dis_{\ge\noudo T''}(V)$).
We map $(T,D)$ to $(T',D')\times(T'',D'')$ in the lax colimit.
Note that $\noudo T'+\noudo T''=\noudo T\ge r$, so this is
well-defined.

In order to verify that the object $(T',D')\times(T'',D'')$ of
the colimit satisfies the required universal property, let another
object of the colimit, $(T'_1,D'_1)\times(T''_1,D''_1)$ be
given, where $(T'_1,D'_1)\in\Dis_{\ge i_1}(U)$,
$(T''_1,D''_1)\in\Dis_{\ge j_1}(V)$, for $i_1$, $j_1$ such that
$i_1\le\noudo T'_1$, $j_1\le\noudo T''_1$ (and $i_1+j_1\ge r$).
Suppose furthermore that we have a map $(T,D)\to(T_1,D_1)$ in
$\Dis_{\ge r}(U\kocoprod V)$, where $T_1=T'_1\kocoprod T''_1$ and
$D_1=D'_1\kocoprod D''_1$.

Recall that such a map was a surjective map $f\colon T\to T_1$ such
that for every $t\in T$, $D_t\sub D_{1,f(t)}$.
It follows that $f$ decomposes uniquely as a map $f'\kocoprod f''\colon
T'\kocoprod T''\to T'_1\kocoprod T''_1$, where $f'$, $f''$ are surjections.

In particular, we have $\noudo T'\ge\noudo T'_1\ge i_1$ and $\noudo
T''\ge j_1$, and the universal property follows immediately.
\end{proof}

\begin{remark}
The unit for the adjunction is an equivalence, while the counit gets
inverted in the (non-lax) colimit.
It follows that the map
\[
\Dis_{\ge r}(U\kocoprod V)\longfrom\colim_{i+j\ge r}\Dis_{\ge
  i}(U)\times\Dis_{\ge j}(V)
\]
is an equivalence.
\end{remark}

\begin{lemma}\label{lem:filtered-descent}
Let $M$ be a manifold.
Then the map
\[
\laxcol_{U\in\Dis(M)}\Dis_{\ge r}(U)\longto\Dis_{\ge r}(M)
\]
has a left adjoint.
In particular, it is cofinal.
\end{lemma}
\begin{proof}
This is obvious.
\end{proof}

\begin{remark}
A similar remark as the remark to
Lemma \ref{lem:filtered-factorisation} applies to this lemma as well.
\end{remark}

For a manifold $M$, define $\ccat{D}^\hitan(M):=\ccat{D}(M)^\hitan$, and
$\ccat{D}_{\ge r}(M):=\ccat{D}(M)_{\ge r}$, with respect to the functor
$\pi_0\colon\ccat{D}(M)\to\Fin$ which takes the connected components
of a disjoint union of disks.
We represent an object of $\ccat{D}^\hitan(M)$ as a pair $(T,D)$ as before.

\begin{lemma}\label{lem:filtered-cofinal-sifted}
For every $r$, the functor $\Dis_{\ge r}(M)\to\ccat{D}_{\ge r}(M)$ is
cofinal.
Moreover, $\ccat{D}_{\ge r}(M)$ is sifted.
\end{lemma}
\begin{proof}
It suffices to prove that for a finite cover
$p\colon\lift{M}\onto M$, the functor $\Dis_{\ge r}(M)\to\ccat{D}_{\ge
  r}(\lift{M})$ given by taking the inverse images under $p$, is
cofinal.
Indeed \cite[Section 5.3.2]{higher-alg}, the first statement is a
special case.
The second is equivalent to that the functor
$\diag\colon\ccat{D}_{\ge r}(M)\to\ccat{D}_{\ge r}(\lift{M})$ is
cofinal for non-empty $\lift{M}$.
This follows if the composite
$\Dis_{\ge r}(M)\xrightarrow{}\ccat{D}_{\ge
  r}(M)\xrightarrow{\diag}\ccat{D}_{\ge r}(\lift{M})$ is cofinal.

The proof of this upgraded statement is similar to the proof of
\cite[Proposition 5.3.2.13 (1)]{higher-alg}.
Namely, we apply Joyal's generalisation of Quillen's theorem A
\cite{topos} to reduce it to the following.
Namely, let $(T_0,D_0)$ be an object of $\ccat{D}_{\ge r}(\lift{M})$,
and denote the defining embedding $\coprod_{T_0}D_0\into\lift{M}$ by
$i$.
Then it suffices to prove that the category
\[
\laxcol_{(T,D)\in\Dis_{\ge r}(M)}\coprod_{f\colon T_0\onto
  T}\Fibre\Bigg[\prod_{t\in
  T}\Emb\Big(\coprod_{f^{-1}(t)}D_0,p^{-1}D_t\Big)\longto\Emb\Big(\coprod_{T_0}D_0,\lift{M}\Big)\Bigg]
\]
has a contractible classifying space, where the fibre is that over
$i$.

It follows from homotopy theory (see Remark \ref{rem:homotopy-theory})
that the classifying space is equivalent to the space
\[
\Fibre\Bigg[\colim_{(T,D)\in\Dis_{\ge r}(M)}\coprod_{f\colon T_0\onto
  T}\prod_{t\in
  T}\Emb\Big(\coprod_{f^{-1}(t)}D_0,p^{-1}D_t\Big)\:\longto\:\Emb\Big(\coprod_{T_0}D_0,\lift{M}\Big)\Bigg].
\]

Furthermore, we obtain that it suffices to prove that the map
\[
\colim_{(T,D)\in\Dis_{\ge r}(M)}\coprod_{f\colon T_0\onto
  T}\prod_{t\in
  T}\Conf(f^{-1}(t),p^{-1}D_t)\;\longto\;\Conf(T_0,\lift{M})
\]
is an equivalence.
See the proof of Lemma \ref{lem:disj-cofinal-d}.

The equivalence follows from applying the generalised
Seifert--van Kampen theorem to the following open cover of
$\Conf(T_0,\lift{M})$.
The cover is indexed by the category $\Dis_{\ge r}(M)_{T_0/}$, and is
given by the functor which associates to $(T,D)$ with a map
$f\colon T_0\onto T$, the open subset $\prod_{t\in
  T}\Conf(f^{-1}(t),p^{-1}D_t)$ of $\Conf(T_0,\lift{M})$.

It is immediate to see that this cover satisfies the assumption for
the generalised Seifert--van Kampen theorem.
\end{proof}

We conclude as follows.
\begin{proposition}\label{prop:filtered-factorisation-algebra}
Let $\cat{A}$ be a symmetric monoidal stable category whose monoidal
multiplication preserves colimits variable-wise.
Assume that $\cat{A}$ is closed under all sequential limits.

Let $A$ be an augmented factorisation algebra in $\cat{A}$ on a
manifold $M$.
Then the augmented filtered prealgebra $F_\dot A$ is a positive
locally constant filtered factorisation algebra.
\end{proposition}
\begin{proof}
The functor $F_\dot A$ on $\Open(M)$ is symmetric monoidal
by Lemma \ref{lem:filtered-factorisation},
Lemma \ref{lem:filtered-cofinal-sifted}, and the assumption that the
monoidal operations preserve colimits variablewise.

It is locally constant by Lemma \ref{lem:filtered-cofinal-sifted}
(cofinality).

It is the left Kan extension from its restriction to $\Dis(M)$, by
Lemma \ref{lem:filtered-descent}.

It is positive by Example \ref{ex:positivity-of-canonical-filtration}.
\end{proof}

\begin{corollary}\label{cor:poincare-for-completed-algebra}
The completion $\complete{F}_\dot A$ of $F_\dot A$ is a positive
locally constant factorisation algebra taking values in the complete
filtered stable category of the complete filtered objects of
$\cat{A}$.
In particular, the Poincar\'e duality
theorem \ref{thm:poincare-duality-complete} applies to
$\complete{F}_\dot A$.
\end{corollary}
\begin{proof}
This is a consequence of Proposition and the general discussion in
\cite[Sections 3.1, 3.2]{local} of monoidal filtration and filtered
objects.
\end{proof}

\subsection{The dual theorems}
\label{sec:variant}
\setcounter{subsubsection}{-1}
Let $\cat{A}$ be a stable category which is given a filtration, and is
closed under sequential \emph{co}limits.
Then $\cat{B}:=\cat{A}^\op$ is a filtered stable category by defining
$\cat{B}_{\ge r}=(\cat{A}^{\le -r})^\op$, where $\cat{A}^{\le
  s}:=\cat{A}^{<s+1}$.

If suspensions and sequential colimits in $\cat{A}$ are bounded with
respect to the filtration of $\cat{A}$, then the arguments we have
made in the previous sections can be applied to $\cat{B}$.

\begin{definition}
Let $\cat{A}$ be a stable category, and let a filtration and a
monoidal structure $\tensor$ (which is exact in each 
variable) be given on $\cat{A}$.
We say that an integer $p$ is an \kore{upper bound} for the monoidal
structure if $\unity$ belongs $\cat{A}^{\le 0}$, and for every
integers $r$, $s$, the monoidal
operation $\tensor\colon\cat{A}^2\to\cat{A}$ takes the full
subcategory $\cat{A}^{\le r}\times\cat{A}^{\le s}$ of the source to
the full subcategory $\cat{A}^{\le r+s+p}$ of the target.

A monoidal structure is said to be \kore{bounded above} if it has an
upper bound.
\end{definition}

Proof of the following is similar to the proof of Proposition
\ref{prop:commuting-tensor-and-cotensor}.
\begin{proposition}
Let $\cat{A}$ be a stable category which is closed under sequential
colimits.
Suppose given a filtration and a monoidal structure on $\cat{A}$, and
assume that $\cat{A}^\op$ is complete with respect to the filtration.
Let
\begin{itemize}
\item $d$ be an upper bound for sequential colimits in $\cat{A}$,
\item $\omega$ be an upper bound for suspensions in $\cat{A}$,
\item $p$ be an upper bound for the monoidal structure.
\end{itemize}

Let $A$ be an augmented associative algebra, and $C$ an augmented
associative coalgebra, both in $\cat{A}$.
Assume that the augmentation ideal of $C$ belongs to $\cat{A}^{<-p}$.
Assume $A$ is \kore{conegative} in the sense that its augmentation
ideal belongs to $\cat{A}^{<-\omega-p}\intersect\cat{A}^{<-p}$.

Let $K$ be a right $A$-module, $L$ an $A$--$C$-bimodule, and let $X$
be a left $C$-module, all bounded above.

Then the canonical map
\[
K\tensor_A(L\cotensor_CX)\longto(K\tensor_AL)\cotensor_CX
\]
is an equivalence.
\end{proposition}

It follows (if the monoidal operations preserves colimits
variable-wise) that the Poincar\'e duality theorem similar to
Theorem \ref{thm:poincare-duality-complete} holds for conegative
factorisation algebras.
(See the proof of Theorem \ref{thm:poincare-duality-complete}.)

\begin{example}
Let $\cat{A}$ be a stable category, and assume $\cat{A}$ has all
sequential colimits.
Then the filtration on the category of filtered objects of $\cat{A}$
from \cite[Section 3.2]{local} makes
its opposite category a \emph{complete} filtered stable category.
All colimits which exist in $\cat{A}$, is bounded by $0$ in the
category of filtered objects.

Moreover, if $\cat{A}$ is symmetric monoidal by operations which are
exact variable-wise, then the monoidal structure on the filtered
objects, described in \cite[Section 3.2]{local},
is bounded above by $0$.
\end{example}

\begin{example}
Let $\lie{g}$ be a dg Lie-algebra over a field (in the usual discrete
sense) of characteristic $0$.
Then the Chevalley--Eilenberg complex
$C_\dot\lie{g}=(\Sym^*(\susp\lie{g}),d)$ (where $\susp=\blank[1]$ is
the suspension functor, and the differential $d$ is the sum of the
internal differential from $\lie{g}$ and the Chevalley--Eilenberg
differential) can be refined to give a filtered chain complex
\[
\cdots\longto\zero\longto\cdots\longto\zero=F_1C_\dot\lie{g}\longto
F_0C_\dot\lie{g}\longto\cdots\longto
F_{-r}C_\dot\lie{g}\longto\cdots,
\]
where $F_{-r}C_\dot\lie{g}:=(\Sym^{\le r}(\susp\lie{g}),d)$, so
$C_\dot\lie{g}=\colim_{r\to\infty}F_{-r}C_\dot\lie{g}$.
\end{example}

$F_*C_\dot$ takes a quasi-isomorphism to quasi-isomorphism, so induces
a functor between the infinity $1$-categories where quasi-isomorphisms
are inverted.
This is since homological algebra implies that the homotopy
cofibre $F_{-r}C_\dot/F_{-r+1}C_\dot$ is given by the quotient in the
strict/discrete sense by the subcomplex, and preserves
quasi-isomorphisms (since the symmetric group has vanishing higher
homology for coefficients over our field).
Note that the Chevalley--Eilenberg differential vanishes on the
quotient.

Moreover, the functor preserves sifted homotopy colimits since the
layers (or the associated graded) as above do.

Note furthermore that $F_*C_\dot$ is obviously a lax symmetric
monoidal functor.
\begin{lemma}
The lax symmetric monoidal structure of $F_*C_\dot$ is in fact a
genuine symmetric monoidal structure.
\end{lemma}
\begin{proof}
Let $\lie{g}$, $\lie{h}$ be dg Lie algebras.
We would like to show that the canonical map
\[
F_*C_\dot\lie{g}\tensor F_*C_\dot\lie{h}\longto
F_*C_\dot(\lie{g}\directsum\lie{h})
\]
is a quasi-isomorphism.

It suffices to show that the map induces equivalence on all layers.
Thus, let $r\ge 0$ be an integer.
Then the ($-r$)-th layer of the target is
$\Sym^r(\susp\lie{g}\directsum\susp\lie{h})$, as has been seen
already.

The ($-r$)-th layer of the source can be seen to be the direct sum
over $i\ge 0$, $j\ge 0$ such that $i+j=r$, of the cofibre of the map
\[
\colim_{\substack{k\le i,\:\ell\le j\\
    k+\ell\le r-1}}F_{-k}C_\dot\lie{g}\tensor
F_{-\ell}C_\dot\lie{h}\longto F_{-i}C_\dot\lie{g}\tensor
F_{-j}C_\dot\lie{h}.
\]
However, by inductive use of homological algebra, this cofibre can be
seen to be $\Sym^i(\susp\lie{g})\tensor\Sym^j(\susp\lie{h})$.

The result follows immediately.
\end{proof}

It follows that if $\lie{g}$ is a locally constant sheaf of dg Lie
algebras on a manifold $M$, then $F_*C_\dot(\lie{g}^!)$ is an
(augmented) filtered dg factorisation algebra.
Since it is negative, the Poincar\'e duality theorem of this section
applies.

\begin{theorem}\label{thm:sheaf-of-lie}
Let $\lie{g}$ be a locally constant sheaf of dg Lie algebras over a
field of characteristic $0$, on a manifold $M$.
Then the (negative) filtered dg precoalgebra $F_*C_\dot(\lie{g})$
on $M$ satisfies excision.
\end{theorem}
\begin{proof}
It suffices to show $F_*C_\dot(\lie{g})=F_*C_\dot(\lie{g}^!)^+$.
However, this can be done in the same way as the computation of
$C_\dot(\lie{g}^!)^+$ in Section \ref{sec:construction}.
\end{proof}



\begin{thebibliography}{99}
\setcounter{enumiv}{-1}
\bibitem{ayala-fran}Ayala, David; Francis, John,
  \emph{Poincar\'e/Koszul duality}, arXiv:1409.2478.

\bibitem{cg}Costello, Kevin; Gwilliam, Owen,
  \emph{Factorization algebras in quantum field theory}, draft
  available at http://www.math.northwestern.edu/\~{}costello/.

\bibitem{glanon}Francis, John, \emph{Factorization homology and
    $E_n$-algebras}, lecture series in the workshop \emph{``Th\'eorie
    des champs, avec un parfum d\'eriv\'e''}, Rencontres
  math\'ematiques de Glanon and Geometry And Physics IX, in Glanon,
  France, July 2011.

\bibitem{francis}Francis, John, \emph{Factorization homology of
  topological manifolds}, arXiv:1206.5522.

\bibitem{gk}Ginzburg, Victor; Kapranov, Mikhail,
  \emph{Koszul duality for operads.} Duke Math.~J.~\textbf{76} (1994),
  no.~1, 203--272.

\bibitem{goodwi}Goodwillie, Thomas G., \emph{Calculus. III. Taylor
    series.} Geom.~Topol.~\textbf{7} (2003), 645--711 (electronic).

\bibitem{gwilliam}Gwilliam, Owen, \emph{Factorization algebras and
    free field theories}, Thesis, Northwestern University, 2012.

\bibitem{topos}Lurie, Jacob, \emph{Higher Topos Theory.}
Annals of Mathematics Studies, \textbf{170}. Princeton University
Press, Princeton, NJ, 2009. xviii+925 pp. ISBN: 978-0-691-14049-0;
0-691-14049-9

\bibitem{tft}Lurie, Jacob, \emph{On the classification of
    topological field theories.} Current developments in mathematics,
  \textbf{2008}, 129--280, Int. Press, Somerville, MA, 2009.

\bibitem{higher-alg}Lurie, Jacob, \emph{Higher Algebra}, available
  at http://www.math.harvard.edu/\~{}lurie/

\bibitem{descent}Matsuoka, T., \emph{Descent properties of the
    topological chiral homology.}  arXiv:1409.6944.

\bibitem{local}Matsuoka, T., \emph{Koszul duality between
    $E_n$-algebras and coalgebras in a filtered category.}
  arXiv:1409.6943.

\bibitem{thesis}Matsuoka, T., \emph{Descent and the Koszul duality for
    locally constant factorisation algebras.} Thesis, Northwestern
  University, 2014.

\bibitem{mcduff}McDuff, D., \emph{Configuration spaces of positive and
    negative particles.} Topology \textbf{14}, 1975, 91--107.

\bibitem{salvato}Salvatore, P., \emph{Configuration spaces with
    summable labels.} Cohomological Methods in Homotopy
  Theory. Progress in Mathematics \textbf{196}, 2001, 375--396.

\bibitem{segal-config}Segal, G., \emph{Configuration-spaces and
    iterated loop-spaces.} Inventiones Math.~\textbf{21}, 1973, no.~3,
  pp.~213--221.

\bibitem{segal-cohomology}Segal, Graeme, \emph{Categories and
    cohomology theories.} Topology \textbf{13} (1974), 293--312.

\bibitem{toen-tannaka}To\"en, Bertrand, \emph{Dualit\'e de Tannaka
    sup\'erieure.} Unpublished MPI preprint, June 2000.

\bibitem{toen-vezzosi}To\"en, Bertrand; Vezzosi, Gabriele, \emph{A
    remark on $K$-theory and $S$-categories.} Topology \textbf{43}
  (2004), no.~4, 765--791.
\end{thebibliography}
\end{document}